\documentclass[11pt,a4paper]{article}
\usepackage[utf8]{inputenc}
\usepackage[T1]{fontenc}
\usepackage[french, english]{babel}
\usepackage{amsfonts}
\usepackage{amsmath, amsthm, mathrsfs}
\usepackage{bbm}
\usepackage{graphicx}
\usepackage[most]{tcolorbox}
\usepackage{xcolor}
\usepackage{pifont}
\usepackage{array}
\usepackage[left=2.5cm,right=2.5cm,top=2.5cm,bottom=2.5cm]{geometry}
\usepackage{caption}
\usepackage{subcaption}

\usepackage[colorlinks=true,linkcolor=purple,citecolor=magenta]{hyperref}

\newcommand{\E}{\mathbb{E}}

\newcommand{\R}{\mathbb{R}}
\newcommand{\Z}{\mathbb{Z}}
\newcommand{\T}{\mathbb{T}}
\newcommand{\un}{\mathbbm{1}}
\newcommand{\PP}{\mathbb{P}}
\newcommand{\bP}{\textbf{P}}
\newcommand{\Q}{\mathbb{Q}}
\renewcommand{\P}{\mathbb{P}}

\newcommand{\FEP}{\textsc{fep}}

\newcommand{\SSEP}{\textsc{ssep}}

\newcommand\ent[1]{\lfloor #1 \rfloor} 
\newcommand\entc[1]{\lceil #1 \rceil}

\newtheorem{theorem}{Theorem}[section]
\newtheorem{lemma}[theorem]{Lemma}
\newtheorem{proposition}[theorem]{Proposition}
\newtheorem{corollary}[theorem]{Corollary}
\newtheorem{definition}{Definition}[section]
\newtheorem{remark}{Remark}[section]
\definecolor{MSBlue}{rgb}{.15,.0.35,.85}

\newcommand{\eqs}[1]{\begin{equation}#1\end{equation}}
\newcommand{\eq}[2]{\begin{equation}\label{#1}#2\end{equation}}

\title{Cutoff for the mixing time of the Facilitated Exclusion Process}
\author{Brune Massoulié}
\date{}
\begin{document}
\maketitle
\begin{abstract}
We compute the mixing time of the Facilitated Exclusion Process (FEP) and obtain cutoff and pre-cutoff in different regimes. The main tool to obtain this result is a new bijective, deterministic mapping between the joint law of an ergodic FEP and its current through the origin, and the joint law of a Simple Symmetric Exclusion Process (SSEP) and its current through the origin. This mapping is interesting in itself, as it remains valid in the non-ergodic regime where it gives a coupling between the position of a tagged particle in the FEP and the current through the origin in a SSEP with traps.
\end{abstract}

\section{Introduction}
The Facilitated Exclusion Process (FEP) is an interacting particle system on a lattice that was introduced in the physics literature by \cite{rossi_universality_2000, basu_activeabsorbing-state_2009} as a model with an active-absorbing-state phase transition. It is defined as follows: there is at most one particle per lattice site and each particle tries to jump to a neighbouring site at rate 1. The jump is not allowed if the target site is occupied (exclusion constraint), or if the particle is isolated, i.e. it has no nearest neighbour occupied site (kinetic constraint). The kinetic constraint makes the FEP's behaviour very different from that of the Simple Symmetric Exclusion Process (SSEP), which is defined similarly but without the kinetic constraint. Indeed, unlike the SSEP, the FEP is not reversible with respect to product measures, non-attractive and has transient configurations and absorbing states (when no particles have neighbours, the system freezes). 

The FEP has mostly been studied on one-dimensional lattices, such as $\Z$ \cite{baik_facilitated_2018, zhao_invariant_2019, goldstein_stationary_2022,erignoux_stationary_2023}, a closed segment \cite{ayre_mixing_2024}, a segment connected to reservoirs \cite{da_cunha_hydrodynamic_2024} and the discrete circle $\T_N:= \Z/N\Z$ \cite{gabel_facilitated_2010,blondel_hydrodynamic_2020,blondel_stefan_2021, ayre_mixing_2024, erignoux_cutoff_2024}.
In this article we will study the mixing time for the FEP on a discrete circle. A first important observation is that,
depending on the number $K$ of particles,  the FEP on $\T_N$ has very different long-time behaviours :
\begin{itemize}
\item  if $K \le \frac{N}{2}$ we say the FEP is \emph{subcritical}. In this case all particles end up becoming isolated and therefore the system ultimately becomes frozen, i.e.  no jumps are possible anymore;
\item if $K > \frac{N}{2}$ we say the FEP is \emph{supercritical}. In this case the system never reaches a frozen configuration, but at some point each empty site becomes surrounded by particles and this property is preserved forever: we then say the system has reached its ergodic component. This is due to the fact that the dynamics can separate 2 neighbouring empty sites but not make them join each other (this would require a jump of an isolated particle).
\end{itemize}

Our main result is to show cutoff for the mixing time of a supercritical FEP on $\T_N$. More precisely, taking $N$ to infinity the $\varepsilon$-mixing time (namely the time needed for a Markov chain to be $\varepsilon$-close to its invariant measure) does not depend  at first order on $\varepsilon$ (see Theorem \ref{theo1}).
This means that the worst total variation distance to equilibrium as a function of time goes very abruptly from almost 1 to almost 0 around the mixing time. The system stays far from equilibrium until just before the mixing time, then becomes almost indistinguishable from equilibrium after a short window of time, which is dominated by the mixing time. Such an abrupt convergence has been viewed as a time-reverse of escape behaviour from some metastable chains \cite{barrera_abrupt_2009}.
This type of phenomena was first found by \cite{diaconis_generating_1981,aldous_shuffling_1986} in the context of card shuffling, and was since shown in many Markov chains (see for an introduction \cite[Chapter 18]{levin_markov_2017}). We mention in particular that cutoff for the SSEP on the circle was shown in \cite{lacoin_simple_2017}, as we will use some of its tools in this article. 
Most previous works on cutoff concern irreducible chains and cutoff is proved by estimating precisely the mixing time of each model. 
A more general theory was introduced recently \cite{salez_cutoff_2023,salez_modern_2025} to prove the occurence of cutoff without needing to compute precisely the mixing time, but this is only valid under curvature assumptions and for irreducible chains. Since the FEP is not irreducible, it is necessary to study precisely its mixing time to show cutoff, which is quite challenging since one must control the \emph{transience time}, meaning the time needed to leave the transient component.

The mixing time of the FEP was previously studied in  \cite{ayre_mixing_2024}  and \cite{erignoux_cutoff_2024}. In \cite{ayre_mixing_2024}  pre-cutoff for the mixing time, and cutoff if restricted to the ergodic component were proved for the FEP on the segment.
For the FEP on the circle,  \cite{ayre_mixing_2024} showed the mixing time is of order $N^2\log N$ with some conditions on the initial configuration. 
In these pioneer results, controlling the transience time was a major issue. In \cite{erignoux_cutoff_2024}, extending to the circle a mapping developed in \cite{ayre_mixing_2024} for the segment, the transience time of the FEP on the circle was precisely computed, which allowed to generalise the mixing time result of \cite{ayre_mixing_2024} to all initial configurations. 
Moreover, \cite{erignoux_cutoff_2024} proves cutoff for the transience time and computes the transience time as a function of the number of particles. 
Our result is a great improvement compared to  what was already known, since we are able to compute the mixing time precisely enough to obtain cutoff: instead of showing that the mixing time is $\mathcal{O}_\varepsilon(N^2\log N)$, we show it is $\frac{N^2}{4\pi^2}\log N (1+o_\varepsilon(1))$. We achieve this by combining estimates on the transience time from \cite{erignoux_cutoff_2024} and a new estimate on the mixing time once in the ergodic component. As was observed for the SSEP with traps in \cite{erignoux_cutoff_2024}, the times spent in the transient phase and in the ergodic component are often not of the same order, so we do not lose too much precision by studying these two phases separately. 
 
An account of our main results and ideas follows. First we prove that the worst mixing time over all initial supercritical configurations exhibits cutoff (see Theorem \ref{theo1}). We also study the mixing time as a function of the number of particles $K$, and obtain bounds that imply cutoff in certain regimes, such as the close-to-critical regime (e.g. $K=\frac{N}{2} + \log N$) and the macroscopically supercritical regime (e.g. $K = \rho N, \, \frac12<\rho<1$), and that imply pre-cutoff in the intermediate case (e.g. $K=\frac{N}{2} + N^\alpha, \,0<\alpha<1$) (see Theorem \ref{th:diff_reg}). Our main tool is a new bijective mapping from the pair of an ergodic FEP and its current through the origin, to the pair of a SSEP and its current through the origin. This is a natural extension of the mapping introduced in \cite{erignoux_cutoff_2024}. By using this mapping together with a height function representation of the SSEP and its current, and adapting the work of \cite{lacoin_simple_2017}, we obtain sharp results on the mixing time of FEP started in the ergodic component. Then we combine this result with the sharp results on the transience time we obtained in \cite{erignoux_cutoff_2024} to get our Theorems \ref{theo1} and \ref{th:diff_reg}.
 
The use of mappings to study the FEP is not novel. We recall the mapping to a zero-range process \cite{basu_activeabsorbing-state_2009, blondel_hydrodynamic_2020}, the interpretation of the ergodic FEP as an exclusion process with objects of size 2 \cite{gabel_facilitated_2010, ayre_mixing_2024}, more recently a lattice path representation \cite{ayre_mixing_2024}, and a mapping to a SSEP with traps \cite{erignoux_cutoff_2024} which is the particle system version of the latter. These processes are more convenient to study than the FEP, mainly because the zero-range process and the SSEP with traps are attractive. The novelty of our mapping is that it is bijective. Let us emphasize that our mapping, besides being a key tool for the study of the mixing time, has other interesting consequences. For instance, by Corollary \ref{cor:mapc}, the position of a tagged particle in the ergodic FEP can be exactly coupled to the current through the origin in an associated SSEP. More generally, the position of a tagged particle in the FEP can be deterministically mapped to the current through the origin of a SSEP with traps. This mapping remains valid by changing jump rates (for example in the asymmetric setting), and can also be defined on the full integer line, so we expect our method to be useful for studying the FEP in broader settings. However, as for the mappings mentioned above, it would be more difficult to do this construction in a non-conservative setting (e.g. a FEP in contact with reservoirs \cite{da_cunha_hydrodynamic_2024}). 

An interesting perspective would be to try and show cutoff in the intermediate case, where we only showed pre-cutoff. In this regime, the transience time and the ergodic mixing time both have the same order, so one would need to understand well the distribution of the FEP once it leaves the transient component to know if it has already started mixing. A challenge for this is that the worst initial conditions for the transience time and the mixing time are not known, although it was conjectured in \cite[Remark 1]{erignoux_cutoff_2024} that a step initial condition should be the worst. Without knowing this, one must control the process started from any initial condition. 

Another extension would be to study more precise notions of cutoff, namely the cutoff window or the cutoff profile, defined for example in \cite[Section 2.2]{barrera_thermalisation_2020}. One could search for sharp windows for the cutoff times, meaning the timescale on which the distance to equilibrium decreases from $1-\varepsilon$ to $\varepsilon$. For the transience time, it was conjectured in \cite[Remark 1]{erignoux_cutoff_2024} that the true window should be of order $N^2$, which would be the case if the worst initial condition was the conjectured one. Such a result would also improve the window for the mixing time. Going even further, one could aim to study cutoff profiles, meaning the shape of the total variation distance curve around the mixing time, rescaling time by the window size. Although ever more results of this kind are being discovered, they are very difficult to obtain in general, see \cite{teyssier_every_2025} for an introduction to this field. However, since this was done for the SSEP on the circle in \cite{lacoin_cutoff_2016}, perhaps using our bijective mapping one could hope for such a result, at least in the ergodic regime. This is by no means evident and would require careful control of the links between the two distributions. 

\subsection*{Notation and conventions}
In this article, we will work with three different processes, that can be mapped into each other: the Facilitated Exclusion Process (FEP), the Simple Symmetric Exclusion Process (SSEP) and the Corner-Flip Dynamics (CFD). Although, thanks to the mappings we will introduce, these three processes can be defined on the same probability space, we will use different notations for their distributions, to highlight which process we are working on. We will also choose different typical names for the configurations and the variables according to which process we are looking at. These conventions are summarised below:
\begin{center}
\def\arraystretch{1.4}
\begin{tabular}{|l|c|c|c|}
\hline
Process & FEP & SSEP & CFD \\
\hline
Total number of sites &$N$ & $K$& $K$ \\
Position of a site & $x$ & $k$ & $k$ \\
Number of particles & $ K$ & $P$ & $P$ uphill slopes \\ 
Typical configuration name & $\eta=(\eta_x)_{x\in \T_N}$ & $\sigma=(\sigma_k)_{k\in\T_K} $ & $\zeta=(\zeta_k)_{k \in \T_K}$  \\
Distribution of the process  & ${\PP}_\eta$  & $\bP_\sigma$ & $\mathbb{Q}_\zeta$ \\
\hline
\end{tabular}
\end{center}
Here are some other conventions we will use:
\begin{itemize}
\item We will write the elements of $\T_N := \Z/N\Z$ as $\{1, ..., N\}$ so we start at $1$ and finish at $N$.
\item We will often consider intervals of sites of the periodic lattice. Throughout the article, these intervals are considered clockwise, more precisely, for $x,y\in \T_N$, if $1\le y < x \le N$, then $[x,y] = \{x, x+1, ... N\} \cup \{1, ... y\}$. It should be clear from context whether we are working modulo $N$ or $K$, so we will not explicit this size parameter in the clockwise intervals. 
\item We will write $f(N) = \mathcal{O}_{\varepsilon}\left(g(N)\right)$ to indicate that there exists a constant $C_\varepsilon >0$, depending on $\varepsilon$, such that for all $N$, $f(N) \le C_\varepsilon g(N)$.
\end{itemize}

\section{Model and results}
We study the Facilitated Exclusion Process (FEP) on the discrete circle $\T_N = \{1,... N\}$. It is a particle system with an exclusion constraint, so on each site there is at most 1 particle. Therefore, its configurations belong to $\Gamma_N := \{0,1\}^{\T_N}$, and we denote a configuration by $\eta = (\eta_x)_{x \in \T_N}$, where $\eta_x = 1$ if there is a particle at site $x$ and $\eta_x=0$ if site $x$ is empty. 

The FEP has the following dynamics: each particle tries jumping at rate 2 to the left or to the right with probability $\frac12$ (equivalently, each jump direction is attempted at rate 1). However, the jump is cancelled if one of the two following constraints is not satisfied:
\begin{itemize}
\item The exclusion constraint: if the target site is already occupied, the jump is forbidden.
\item The kinetic constraint: if, before the jump, the particle is isolated (meaning it has no neighbour), the jump is forbidden.
\end{itemize}

In other words, the FEP on $\T_N$ is the continuous time Markov process on $\Gamma_{N}$ with generator given by
\eqs{\mathcal{L}_N^{\FEP}f(\eta) = \sum_{x \in \T_N}\sum_{z = \pm 1}\eta_x \eta_{x-z} (1 - \eta_{x+z}) (f(\eta^{x,x+z}) - f(\eta)),}
where
\begin{equation} \label{eq:swap-fep}
\eta_y^{x,x+z} =
    \begin{cases}
        \eta_y \hbox{ if } y \notin \{x,x+z\}\\
        \eta_{x+z} \hbox{ if } y=x \\
        \eta_x \hbox{ if } y = x+z,
    \end{cases}       
    \end{equation}
which corresponds to a jump from site $x$ to $x+z$ if $\eta_{x-z}\eta_x(1-\eta_{x+z})=1$. 

Furthermore, the system is conservative (the number of particles is preserved by the dynamics), so it will be convenient to study the FEP for a given number of particles $K$. We thus define the set of exclusion configurations on $\T_N$ with $K$ particles: 
\eq{eq:def_gammaNK}{\Gamma_{N,K} = \{\eta \in \{0,1\}^{\T_N} : |\eta| = K\}.}

Depending on the number of particles $K$, the FEP has different long-time behaviours. If $K \le \frac{N}{2}$ (the subcritical case), the system ends up becoming frozen: all particles are isolated and no jumps are possible. On the other hand, if $K > \frac{N}{2}$ (the supercritical case), the system never freezes, but eventually reaches an absorbing set of configurations, the ergodic component (the set of configurations with every empty site surrounded by particles). 

Therefore, the FEP is not irreducible and has transient states: in the subcritical case, they are the configurations that are not frozen, and in the supercritical case, they are the configurations that are not ergodic (with at least one pair of neighbouring empty sites). 

In this article, we study the FEP's mixing time, and therefore focus on the regime $K > \frac{N}{2}$, which is the only one where a non frozen stationary state is reached. We define, for $K > \frac{N}{2}$,
\eqs{\mathcal{E}_{N,K} = \{\eta \in \Gamma_{N,K} : \forall x \in \T_N, \eta_x + \eta_{x+1} \ge 1\}}
the set of ergodic configurations of the FEP. Set also 
\eqs{\mathcal{T}_{N,K} = \Gamma_{N,K} \setminus \mathcal{E}_{N,K}}
the set of transient configurations of the FEP on $\T_N$ with $K$ particles. 
Then, the invariant law of the FEP on $\T_N$ with $K$ particles is given by
\eqs{\pi_{N,K}^{\FEP} = \mathcal{U}(\mathcal{E}_{N,K}),}
with $\mathcal{U}(A)$ denoting the uniform measure over a set $A$.

For $\eta \in \Gamma_{N,K}$, we denote by $\P_\eta(\eta(t) \in \cdot)$ the distribution of the FEP started from $\eta$ at time $t$. We also write $d_{\textsc{tv}}$ for the total variation distance. Then we define for all $\varepsilon \in (0,1)$ the $\varepsilon$-mixing time 
\eq{eq:deftmix}{\tau^{\FEP}_{N,K}(\varepsilon) = \inf \{t \ge 0 : \forall \eta \in \Gamma_{N,K}, d_{\textsc{tv}}\left(\P_\eta(\eta(t) \in \cdot), \pi_{N,K}^{\FEP}\right) \le \varepsilon \}, }
the first time such that for any initial configuration $\eta\in\Gamma_{N,K}$, the distribution of $\eta(t)$ is at distance less than $\varepsilon$ from $\pi_{N,K}^\FEP$. 
Our main results concern the behaviour of $\tau^{\FEP}_{N,K}(\varepsilon)$ as $N$ and $K$ go to infinity. We define here the cutoff and pre-cutoff phenomena.
\begin{definition}[Cutoff and pre-cutoff] A sequence of mixing times $(\tau_N)_{N\ge 0}$ exhibits \emph{cutoff} if
\eqs{\forall \varepsilon \in (0,1), \, \lim\limits_{N\to \infty} \frac{\tau_N(\varepsilon)}{\tau_N(\frac{1}{4})} = 1.}
A sequence of mixing times $(\tau_N)_{N\ge 0}$ exhibits \emph{pre-cutoff} if there exist $C, C' >0$ such that
\eqs{\forall \varepsilon \in (0,1), \, C \le \varliminf\limits_{N\to \infty} \frac{\tau_N(\varepsilon)}{\tau_N(\frac{1}{4})}\le \varlimsup\limits_{N\to \infty} \frac{\tau_N(\varepsilon)}{\tau_N(\frac{1}{4})} \le C'.}
\end{definition}
In the cutoff phenomenon, at first order, the mixing time does not depend on $\varepsilon$. The times for the total variation distance to decrease to $1-\varepsilon$ or to $\varepsilon$ are equivalent, so that the total variation distance falls abruptly from almost $1$ to almost $0$ in a small window around $\tau_N(\frac{1}{4})$, which is taken as a reference mixing time. The choice $\frac{1}{4}$ is often made for the reference mixing time, see for example \cite[Section 4.5]{levin_markov_2017}, but is not very important for the definition of cutoff. Pre-cutoff is a weaker property. We now state our first result.
\begin{theorem}[$K$-uniform cutoff for the mixing time]\label{theo1}
For all $\varepsilon \in (0,1)$, there exists $C_\varepsilon >0$ such that for all $N$, 
\eqs{\left|\max\limits_{\frac{N}{2}< K < N} \tau^{\FEP}_{K,N}(\varepsilon) - \frac{1}{4\pi^2}N^2 \log N\right| \le C_\varepsilon N^2\log\log N.}
\end{theorem}
This is an estimate on the mixing time starting from the worst configuration, over all possible values of $K$. We establish that this worst mixing time is asymptotically equivalent to $\frac{1}{4\pi^2}N^2 \log N$, which does not depend on $\varepsilon$, and therefore this sequence of times exhibits cutoff.
\begin{remark}Interestingly, the dominant term of this worst mixing time is exactly the critical FEP's transience time from \cite{erignoux_cutoff_2024}. In fact, the worst mixing time is achieved in the close-to-critical regime, when $K-N/2$ is small, and the transience time is very long and dominates.
\end{remark}

\begin{remark}
Notice that this worst mixing time uniform in $K$ is not the most usual notion of mixing time: mixing times of particle systems are often studied for a sequence of systems of size $N$ with $K(N)$ particles as $N$ goes to infinity. 
The notion used in Theorem \ref{theo1} is analogous to the worst case mixing time shown for the SSEP with traps in \cite[Theorem 1.3]{erignoux_cutoff_2024} and allows for a compact formulation. 
\end{remark} 

The following theorem addresses the mixing time as a function of $K$.
\begin{theorem}[Cutoff and pre-cutoff as a function of $K$] \label{th:diff_reg}
For all $\varepsilon \in (0,1)$, for all sequence $K=K(N)$ such that for all $N$, $N/2 < K < N$, 
\begin{itemize}
\item If $\frac{\log (2K-N)}{\log K} \longrightarrow 0$ (e.g. $K = N/2 + \log N$), then 
\eqs{\left|\tau^{\FEP}_{N,K}(\varepsilon)-\frac{1}{\pi^2}K^2 \log K\right| =\mathcal{O}_\varepsilon\left( N^2 \max\left(\log\log N,  \log (2K-N)\right)\right).}
\item If $\frac{\log (2K-N)}{\log K} \longrightarrow \alpha \in (0,1)$ (e.g. $K = N/2 + N^\alpha$), then
\eqs{\max\Big(\frac{1-\alpha}{\pi^2},\frac{\alpha}{8\pi^2}\Big) \le \varliminf\limits_{N \to \infty} \frac{\tau^{\FEP}_{N,K}(\varepsilon)}{K^2 \log K} \le \varlimsup\limits_{N \to \infty} \frac{\tau^{\FEP}_{N,K}(\varepsilon)}{K^2 \log K} \le \frac{1-\alpha}{\pi^2} + \frac{\alpha}{8 \pi^2}.}
\item If $\frac{\log (2K-N)}{\log K} \longrightarrow 1$ and $(N-K)\to \infty$  (e.g. $K/N \longrightarrow \rho \in (\frac12,1)$), then, setting $P=\min(2K-N,N-K)$,
\eqs{\left|\tau^{\FEP}_{N,K}(\varepsilon)- \frac{1}{8\pi^2}K^2 \log P\right| =\mathcal{O}_\varepsilon \left(K^2\left(1+\log\frac{K}{2K-N}\right)\right).}
\end{itemize}
\end{theorem}

The FEP dynamics can be decomposed in 2 phases: first the system needs to reach the ergodic component, this is the \emph{transience time}, then once ergodic the FEP continues evolving to become mixed, we call this the \emph{ergodic mixing time}. Our main contribution consists in estimating the ergodic mixing time. The three regimes above can then be understood in the following way. When $K$ is close to $N/2$, the transience time dominates the ergodic mixing time, and the cutoff follows from the transience time cutoff from \cite{erignoux_cutoff_2024}. When $K-N/2$ is for example a power of $N$, both times have the same order so we only obtain pre-cutoff. Last, when $K-N/2$ is for example a positive fraction of $N$, the ergodic mixing time dominates the transience time, and corresponds to a related SSEP's mixing time. Then the cutoff follows from adapting the proof of \cite{lacoin_simple_2017} for the cutoff of the SSEP.

\begin{remark}
The previous result proves cutoff and pre-cutoff for most behaviours of $K$. In the case where $(N-K)$ is bounded, we can show that the mixing time is $\mathcal{O}_{\varepsilon}(N^2)$. For this case, there may be no cutoff at all, as for the simple random walk or for the SSEP with a finite number of particles, see \cite[Remark 1.3]{lacoin_simple_2017}. Showing this for the FEP would require some additional control on the distribution when the system leaves the transient component, and is thus left for future work. 
\end{remark}

\section{Structure of the proof} \label{sec:not_struc}
For $N \ge 2$ and $N/2<K<N$, we set for all $\varepsilon \in (0,1)$,
\eqs{\theta_{N,K}(\varepsilon) = \inf \big\{t \ge 0 : \forall \eta \in \Gamma_{N,K}, \P_{\eta}(\eta(t) \in \mathcal{T}_{N,K}) \le \varepsilon\big\},}
the transience time of the FEP on $\T_N$ with $K$ particles. This is the first time such that for any initial configuration, the process has reached the ergodic component with probability greater than $1-\varepsilon$.
We also define
\eqs{\tau^{\FEP}_{\mathcal{E}_{N,K}}(\varepsilon) = \inf \big\{t \ge 0 : \forall \eta \in \mathcal{E}_{N,K}, d_{\textsc{tv}}(\P_\eta(\eta(t) \in \cdot), \pi_{N,K}^{\FEP}) \le \varepsilon \big\}}
the mixing time of a FEP started from $\mathcal{E}_{N,K}$, and call this the ergodic mixing time. 
This quantifies the time to converge to $\pi_{N,K}^{\FEP}$ when started from a configuration that is already ergodic, whereas the full mixing time $\tau^{\FEP}_{N,K}(\varepsilon)$ defined in \eqref{eq:deftmix} is the time such that, started from any initial configuration, including transient ones, the law of the process is sufficiently close to $\pi_{N,K}^{\FEP}$.

We first bound the full mixing time in terms of the transience time and the ergodic mixing time. 
\begin{proposition}\label{prop:easy} For all $\varepsilon \in (0,1), N \ge 2$ and $N/2 < K < N$,
\eq{eq:encadre}{\max\left(\theta_{N,K}(\varepsilon),\tau^{\FEP}_{\mathcal{E}_{N,K}}(\varepsilon)\right)\le \tau^{\FEP}_{N,K}(\varepsilon) \le \theta_{N,K}\left(\frac{\varepsilon}{2}\right) + \tau^{\FEP}_{\mathcal{E}_{N,K}}\left(\frac{\varepsilon}{2}\right).}
\end{proposition}
\begin{proof} Clearly, $\tau_{N,K}^{\FEP}(\varepsilon)$ is greater than $\theta_{N,K}(\varepsilon)$, because under $\pi_{N,K}^{\FEP}$ the probability to be transient is zero. One can also lower-bound $\tau_{N,K}^{\FEP}(\varepsilon)$ by $\tau^{\FEP}_{\mathcal{E}_{N,K}}(\varepsilon)$ because the configurations considered in the definition of $\tau^{\FEP}_{\mathcal{E}_{N,K}}(\varepsilon)$ are a subset of those for $\tau^{\FEP}_{N,K}(\varepsilon)$. 
Last, the upper bound is a consequence of Markov property.
\end{proof}

To control $\tau_{N,K}^{\FEP}$, we will therefore combine estimates on $\theta_{N,K}$ and on $\tau^{\FEP}_{\mathcal{E}_{N,K}}$. We will use the following bounds on the transience time from \cite{erignoux_cutoff_2024}:

\begin{proposition}[Transience time estimate]\label{prop:t} \cite[Lemma 4.5, Remark 9]{erignoux_cutoff_2024} For all $\varepsilon \in (0,1)$, there exists $C_{\varepsilon}>0$ such that for all $N\ge 2$ and $N/2<K<N$,
\eq{eq:transience}{ -C_{\varepsilon} K^2 \le \theta_{N,K}(\varepsilon) - \frac{K^2}{\pi^2} \log \frac{K}{2K-N}\le C_{\varepsilon} K^2 \log\log K.}
If $K = K(N)$ is such that $\frac{\log (2K-N)}{\log K} \longrightarrow 1$, and $\log\frac{K}{2K-N}=o(\log\log K)$,
\eq{eq:transience_short}{\theta_{N,K}(\varepsilon) =\mathcal{O}_{\varepsilon} \left(K^2 \log \frac{K}{2K-N}\right).}
\end{proposition}
The missing ingredient for the mixing time, and the main contribution of this work, is to obtain precise bounds on $\tau_{\mathcal{E}_{N,K} }^{\FEP}(\varepsilon)$.

We set, for all $N$, 
\eqs{a_N = \sqrt{\max(1, \log\log N)}.} It will play the role of separating the regimes of $K$ far from $N$ or $N/2$, and $K$ close to these bounds. Any choice of $(a_N)$ going to infinity and less than $\sqrt{\log\log N}$ would be suitable. Section \ref{sec:u} is devoted to proving the following upper bounds. 
\begin{proposition}[Upper bound, far from the edges] \label{prop:UB}
For all $\varepsilon \in (0,1)$, for all sequence $K=K(N)$ such that for all $N$, $N/2 + a_N < K < N - a_N$, there exists $C_\varepsilon > 0$ such that
\eq{eq:propUB}{\tau^{\FEP}_{\mathcal{E}_{N,K}}(\varepsilon) \le \frac{K^2}{8\pi^2}\log\min(2K-N,N-K) + C_\varepsilon K^2.}
\end{proposition}
\begin{proposition}[Upper bound, close to the edges] \label{prop:UBc}
For all $\varepsilon \in (0,1)$, for all sequence $K=K(N)$ such that for all $N$, $N/2 < K \le N/2 + a_N$ or $N > K \ge N - a_N$, there exists $C_\varepsilon > 0$ such that
\eq{eq:propUBc}{\tau^{\FEP}_{\mathcal{E}_{N,K}}(\varepsilon) \le  C_\varepsilon (Ka_N)^2.}
\end{proposition}
\begin{remark}
In the conditions of Proposition \ref{prop:UB},  by \cite{lacoin_simple_2017}, the right hand side of \eqref{eq:propUB} is equivalent to the $\varepsilon$-mixing time of a SSEP on $\T_K$ with $2K-N$ particles. 
\end{remark}

Section \ref{sec:b} is devoted to proving the following lower bound.

\begin{proposition}[Lower bound] \label{prop:LB}
Set $P = \min(2K-N,N-K)$. If $P \to \infty$, for all $\varepsilon \in (0,1)$, there exists $C_\varepsilon >0$ such that for all $N \ge 2$, 
\eqs{\label{eq:lb-ergo}\frac{K^2}{8\pi^2}  \log P  - C_\varepsilon K^2  \le \tau_{\mathcal{E}_{N,K} }^{\FEP}(\varepsilon).}
\end{proposition}

\begin{remark}
Under the assumptions of Proposition \ref{prop:LB}, the left hand side of \eqref{eq:lb-ergo} corresponds to the mixing time of a SSEP on $\T_K$ with $2K-N$ particles.
\end{remark}
We now prove the theorems by combining \eqref{eq:encadre} and the adequate upper and lower bounds from the propositions.
\begin{proof}[Proof of Theorem \ref{theo1} and \ref{th:diff_reg}]
If $\frac{\log (2K-N)}{\log K} \to 0$, we take the left hand side of \eqref{eq:transience} as a lower bound:
\begin{align}
\tau_{N,K}^{\FEP}(\varepsilon) &\ge \frac{K^2}{\pi^2} \log \frac{K}{2K-N} - C_{\varepsilon}K^2 \nonumber \\
& \ge \frac{K^2}{\pi^2} \log K - C'_{\varepsilon}K^2 (1+\log (2K-N)). \label{eq:lbcrit}
\end{align} 
For the upper bound, we sum the right hand side of \eqref{eq:transience} to a suitable upper bound on $\tau_{\mathcal{E}_{N,K}}(\varepsilon)$. Up to extracting a subsequence, either $K > \frac{N}{2} + a_N$ for all $N$ sufficiently large, or $K \le \frac{N}{2} + a_N$ for all $N$ sufficiently large. Applying \eqref{eq:propUB} in the first case and \eqref{eq:propUBc} in the second case, we obtain:
\eqs{\tau_{N,K}^{\FEP}(\varepsilon) \le \frac{K^2}{\pi^2}\log K + \mathcal{O}_{\varepsilon}\left(K^2 \max\left(\log \log K, \log(2K-N)\right)\right).}
This proves the first point of Theorem \ref{th:diff_reg}.

We now prove the second point of Theorem \ref{th:diff_reg}. If $\frac{\log(2K-N)}{\log K} \to \alpha \in (0,1)$, we use \eqref{eq:transience} and \eqref{eq:lb-ergo} for the lower bound:
\eqs{\frac{\tau_{N,K}^{\FEP}(\varepsilon)}{K^2\log K} \ge \max\left(\frac{1-\alpha}{\pi^2}  + o(1) , \frac{\alpha}{8\pi^2} + o(1) \right).}
We sum \eqref{eq:transience} and \eqref{eq:propUB} to obtain the upper bound:
\eqs{\frac{\tau_{N,K}^{\FEP}(\varepsilon)}{K^2\log K} \le \frac{1-\alpha}{\pi^2} + \frac{\alpha}{8\pi^2} + o(1).}

For the third point of Theorem \ref{th:diff_reg}, assume that $\frac{\log (2K-N)}{\log K} \to 1$ and $(N-K) \to \infty$. We use this time \eqref{eq:lb-ergo} for the lower bound:
\eqs{\tau_{N,K}^{\FEP}(\varepsilon) \ge \frac{K^2}{8\pi^2} \log \min (2K-N,N-K) - C_\varepsilon K^2 .}
For the upper bound, up to choosing a sequence $(a_N)$ growing to infinity with $a_N<N-K$, the assumptions of Proposition \ref{prop:UB} are satisfied, so we can use \eqref{eq:propUB} and sum this with the upper bound from \eqref{eq:transience_short}: 
\eqs{\tau_{N,K}^{\FEP}(\varepsilon) \le \frac{K^2}{8\pi^2}\log \min(2K-N,N-K) + \mathcal{O}_\varepsilon\left(K^2\left(1+\log \frac{K}{2K-N}\right)\right).}

We now prove Theorem \ref{theo1}. First, by \eqref{eq:transience},
\eqs{\sup\limits_{N/2< K< N} \tau_{N,K}^{\FEP}(\varepsilon) \ge \tau_{N,\entc{N/2}+1}^{\FEP}(\varepsilon)\ge \frac{N^2}{4\pi^2}\log N - C_\varepsilon N^2.}
Then, combining \eqref{eq:transience}, \eqref{eq:propUB} and \eqref{eq:propUBc}, we obtain an upper bound valid for all sequence $K=K(N)$. For all $\varepsilon \in (0,1)$, for all sequence $K(N)$ such that, for all $N$, $N/2 < K < N$, there exist $C^{(1)}_\varepsilon,C^{(2)}_\varepsilon,C^{(3)}_\varepsilon ,C_\varepsilon>0$ such that for all $N$, 
\begin{align}
\tau_{N,K}^{\FEP}(\varepsilon) \le&  \frac{K^2}{\pi^2} \log \frac{K}{2K-N} + C^{(1)}_\varepsilon K^2 \log\log K \nonumber \\
&+ C^{(2)}_\varepsilon K^2 \log\log K \un_{\{\min(2K-N,N-K )\le a_N\}} \nonumber \\
&+ \left(\frac{K^2}{8\pi^2}\log \min (2K-N,N-K) + C_\varepsilon^{(3)}K^2\right) \un_{\{\min(2K-N,N-K )> a_N\}} \nonumber\\
\le & \frac{K^2}{\pi^2} \log\frac{K}{2K-N} + \frac{K^2}{8\pi^2} \log \min(2K-N,N-K) + C_\varepsilon N^2\log\log N. \label{eq:ub-all-K}
\end{align}
We apply this to a a sequence $K^*=K^*(N)$ that maximises $\tau_{K,N}(\varepsilon)$ for all $N$, and notice that the maximal value in the right hand side of \eqref{eq:ub-all-K} is achieved for $K$ close to $N/2$. 
\end{proof}
The core of our article will be proving Proposition \ref{prop:UB} in Section \ref{sec:u} and Proposition \ref{prop:LB} in Section \ref{sec:b}, and an important tool for both proofs will be a mapping, developed in Section \ref{sec:map}.

\section{A crucial mapping} \label{sec:map}
We introduce here a mapping between the ergodic FEP and the SSEP which is a key ingredient for the proof of both Propositions \ref{prop:UB} and \ref{prop:LB}. We define it for the ergodic FEP but it can directly be extended to all FEP configurations, that are then mapped to an appropriate SSEP with traps configuration, as exploited in \cite{erignoux_cutoff_2024}.

We first define the SSEP on $\T_K$ with $P$ particles: its configurations belong to $\Gamma_{K,P}$ (defined in \eqref{eq:def_gammaNK}) and each particle tries to jump left or right at rate 1, with an exclusion constraint. In other words, the SSEP is the Markov process on $\Gamma_{K,P}$ with the following generator:
\eqs{\mathcal{L}^{\textsc{SSEP}}_{K,P} f (\sigma) = \sum_{k\in \T_K} \left(\sigma_k(1-\sigma_{k+1}) + (1-\sigma_k)\sigma_{k+1}\right) \left(f(\sigma^{k,k+1}) - f(\sigma)\right),}
where, as in \eqref{eq:swap-fep}, $\sigma^{k,k+1}$ is defined as the configuration obtained after exchanging the values of sites $k$ and $k+1$. 
We will now introduce a static mapping connecting ergodic FEP configurations to appropriate SSEP configurations in a bijective way. 
\paragraph{Static mapping} To ensure bijectivity, 
we will not simply associate a SSEP configuration to a FEP configuration, but a couple made of the position of a tagged particle in the FEP and a SSEP configuration, to the couple made of the rank of this tagged particle in the FEP and a FEP configuration. More precisely, for all FEP configuration $\eta \in \Gamma_{N,K}$, for all $k \in \T_K$, set
\eqs{x_k(\eta) = \inf\{x \in \T_N : \sum_{y=1}^x\eta_y = k \},}
the position of the $k^{\hbox{\tiny th}}$ particle in $\eta$. 
We can then define the static mapping, illustrated in Figure \ref{fig:stat}.
\begin{proposition}[Static mapping] For all $\eta\in \mathcal{E}_{N,K}$, for all $k,l\in \T_K$, set
\eq{eq:sigma_mapping}{\sigma^{(k,\eta)}_l = 2 - x_{k+l}(\eta) + x_{k+l-1}(\eta).}
Then, $\sigma^{(k,\eta)} \in \Gamma_{K,2K-N}$, and the map
\begin{equation} \label{eq:mapping}\Phi : 
\begin{cases} \T_{K} \times \mathcal{E}_{N,K} \to\T_{N}\times \Gamma_{K,2K-N} \\
 (k,\eta) \mapsto  (x_k(\eta), \sigma^{(k,\eta)})
\end{cases}
\end{equation}
is bijective. 
\end{proposition}
Notice that the indices in \eqref{eq:sigma_mapping} belong to $\T_K$, so the sums are to be understood modulo $K$.
\begin{figure}
\centering
\includegraphics[width=0.9\textwidth]{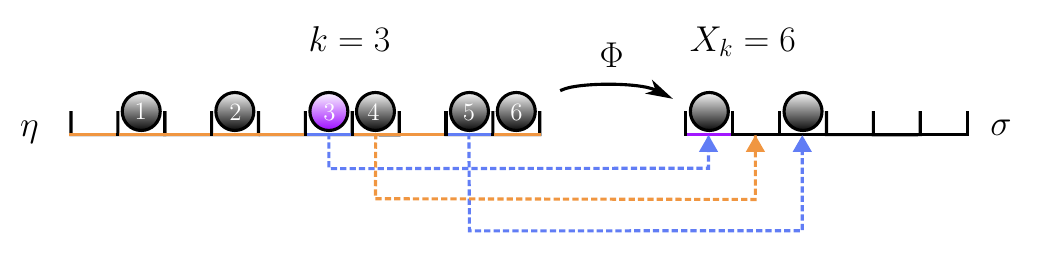}
\caption{Illustration of the static mapping. The rank of the purple particle in $\eta$ is $k$ and its position is $X_k$. Each site of $\sigma$ is in correspondence with a particle of $\eta$: the first site of $\sigma$ is related to the $k^{th}$ particle of $\eta$, the second site of $\sigma$ to the $(k+1)^{th}$ particle of $\eta$, etc.  If a particle of $\eta$ is followed by another particle, it is underlined in blue in $\eta$, and there is a particle on the corresponding site of $\sigma$. If a particle of $\eta$ is followed by an empty site, it is underlined in orange in $\eta$, and the corresponding site of $\sigma$ is empty. }
\label{fig:stat}
\end{figure}
\begin{proof}
\textit{Definition. }It is easy to check that $\sigma^{(k,\eta)}$, as defined by \eqref{eq:sigma_mapping}, is an exclusion configuration: by ergodicity of $\eta$, $1 \le x_j(\eta) - x_{j-1}(\eta) \le 2$ for all $j \in \T_K$. To see that it has $2K-N$ particles, we can sum \eqref{eq:sigma_mapping} for $1 \le l \le K$.

\textit{Surjectivity. }Now we show that, given a couple $(x, \sigma) \in \T_N \times \Gamma_{K,2K-N}$, we can find a couple $(k, \eta) \in \T_K \times \mathcal{E}_{N,K}$ such that $\Phi(k,\eta) = (x, \sigma)$. We construct a set of sites in $\T_N$ 
\eqs{I = \Big\{x + \sum_{j = 1}^l (2- \sigma_j) \mod N, 0\le l \le K-1\Big\},}
where by convention the empty sum is zero. We set $\eta$ the configuration on $\T_N$ such that its occupied sites are the sites of $I$, i.e. $\eta_y = \un_{\{y \in I\}}$ for $y \in \T_N$, and finally set 
\eqs{k = \sum_{y=1}^x \eta_y.}

Now, we show that $\eta \in \mathcal{E}_{N,K}$.
We first check that $\eta$ has $K$ particles, by proving that the elements of $I$ are distinct: for all $0 \le j < l \le K$, 
\eq{eq:surject}{x + \sum_{i = 1}^l (2-\sigma_i) - \left(x + \sum_{i = 1}^j (2-\sigma_i)\right)  = \sum_{i = j+1}^l (2-\sigma_i).}
The right hand side of \eqref{eq:surject} is clearly positive, we now show that it is strictly less than $N$:
\begin{itemize}
\item If $l-j \le N-K$, \eqref{eq:surject} is upper-bounded by $2(l-j) \le 2(N-K) < N$. 
\item If $l-j > N-K$, notice that $\sigma$ has $N-K$ empty sites, so $\sigma$ has at least $l-j - (N-K)$ particles in the segment $[j+1,l]$. Hence \eqref{eq:surject} is upper-bounded by $2(l-j) - (l-j-(N-K)) = l-j + N-K < N$ because $l-j < K$. 
\end{itemize}
In all cases, the difference between two positions indexed by different $j$ and $l$ is strictly between 0 and $N$. Thus, the positions modulo $N$ are still distinct, and $I$ has cardinality $K$.
Finally, $\eta$ is ergodic because by the definition of $I$, the maximum distance between two consecutive particles is 2. 

Hence $(k,\eta) \in \T_K \times\mathcal{E}_{N,K}$, and we just need to show that $\Phi(k,\eta) = (x,\sigma)$. 
First, $\sum_{y=1}^x \eta_y = k$ so $x_k(\eta) \le x$. Since $x\in I$, we have  $\eta_x = 1$, and for all $1 \le y \le x-1, \sum_{z=1}^y \eta_y < k$. Therefore $x_k(\eta) = x$. Similarly, for all $l \in \T_K$, 
\eqs{x_{k+l}(\eta) = x+ \sum_{j=1}^l (2-\sigma_j).} 
We then obtain, for all $l \in \T_K$, 
\eqs{\sigma_l^{(k,\eta)} = 2 - x_{k+l}(\eta) + x_{k+l-1}(\eta) = 2 -(2 - \sigma_{l}) = \sigma_l.}
This concludes the proof that $\Phi(k,\eta) = (x,\sigma)$.

\textit{Injectivity. }Let $(k,\eta), (k', \eta')$ such that $\Phi(k,\eta) =  \Phi(k', \eta')$. Then, 
\begin{align}
&x_k(\eta) = x_{k'}(\eta') \\
&\forall l \in \T_K, \quad x_{k+l}(\eta)-x_{k+l-1}(\eta) = x_{k'+l}(\eta')-x_{k'+l-1}(\eta').
\end{align}
Therefore, for all $l \in \T_K$, $x_{k+l}(\eta) = x_{k'+l}(\eta')$, so $\eta$ and $\eta'$ have the same particle positions, hence $\eta = \eta'$. Then, since $x_k(\eta) = x_{k'}(\eta')$, $k=k'$.
\end{proof}

Notice that this static mapping allows to easily compute the cardinality of $\mathcal{E}_{N,K}$, which was already known by \cite{gabel_facilitated_2010}: 
\eqs{|\mathcal{E}_{N,K}| = \frac{N}{K} |\Gamma_{K,2K-N}| = \frac{N}{K} \binom{K}{2K-N}.}
We could thus have shown only injectivity or surjectivity and used the cardinality of $\mathcal{E}_{N,K}$ to obtain the bijectivity of $\Phi$, but we chose this approach as it yields a nice proof of the cardinality of $\mathcal{E}_{N,K}$.

\paragraph{Dynamic mapping} We now consider the effect of the FEP dynamics on the mapping. We first introduce some notation. 
Let $(k, \eta) \in \T_K \times \mathcal{E}_{N,K}$, we set $k(0) = k$ and $\eta(0) = \eta$. Let $x = x_k(\eta)$, we set $X(0) = x$ and consider the joint dynamics of the FEP and a tagged particle, started from $(x, \eta)$, given by the following generator:
\begin{multline}
\mathcal{L}^{\FEP, tag}f(x,\eta) = \sum_{\substack{1\le y \le N\\ y \neq x}}\eta_y \sum_{z \in \{-1,1\} } \eta_{y-z}(1-\eta_{y+z}) \left(f(x, \eta^{y,y+z}) - f(x, \eta)\right) \\+ \eta_x \sum_{z \in \{-1,1\} } \eta_{x-z}(1-\eta_{x+z}) \left(f(x + z, \eta^{x,x+z}) - f(x, \eta)\right).
\end{multline}
We denote by $\eta(t)$ the FEP at time $t$ and $X(t)$ the position of the tagged particle in $\eta(t)$ at time $t$. We also set
\eqs{k(t) = \sum_{y=1}^{X(t)} \eta_y(t)}
the rank of the particle at site $X(t)$ in $\eta(t)$. By construction, it is a deterministic function of $\eta(t)$ and $X(t)$. 
Defining the current through an edge as the total number of particles having crossed it from left to right minus the total number of particles having crossed it from right to left, $k(t)-k(0)$ is equal to the current through the edge $(N,1)$ from times 0 to $t$ in $\left(\eta(s)\right)_{0 \le s \le t}$, taken modulo $K$.
Last, we set for all $t\ge 0$
\eqs{(Y(t), \sigma(t)) = \Phi(k(t), \eta(t)).}
A first observation is that for all $t \ge 0$, 
\eqs{Y(t) = X(t).}
Indeed, $k(t)$ is defined as the rank of the tagged particle, whose position is $X(t)$, so it is clear that $Y(t) = x_{k(t)}(\eta(t)) = X(t)$ at all times $t\ge 0$.

Then we have the following properties.
\begin{proposition}[Properties of the dynamic mapping] \hfill
\begin{enumerate}
\item The process $(\sigma(t))_{t \ge 0}$ is a SSEP started from $\sigma(0)$.
\item For all $t \ge 0$, $X(t)-X(0)$ is equal to the total current through the edge $(K,1)$ from times 0 to $t$ in $(\sigma(s))_{0 \le s \le t}$, modulo $N$.
\end{enumerate}
\end{proposition}
\begin{figure}
    \centering
    		\begin{subfigure}[t]{7.62cm} 
		\vspace{0pt}
        \includegraphics[width=7.62cm]{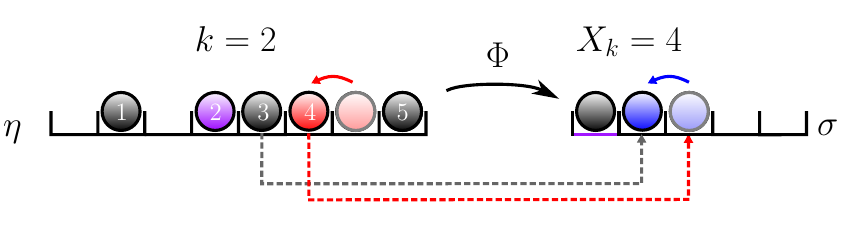}
        \caption{Jump to the left}\label{fig:saut_gauche}
    \end{subfigure}
        \hfill
    \begin{subfigure}[t]{7.62cm}
        \vspace{0pt}
        \includegraphics[width=7.62cm]{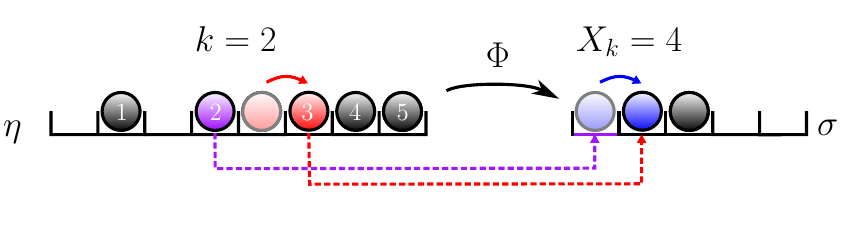}
        \caption{Jump to the right}\label{fig:saut_droite}
    \end{subfigure}
     \begin{subfigure}[t]{7.62cm} 
     \vspace{0pt}
        \includegraphics[width=7.62cm]{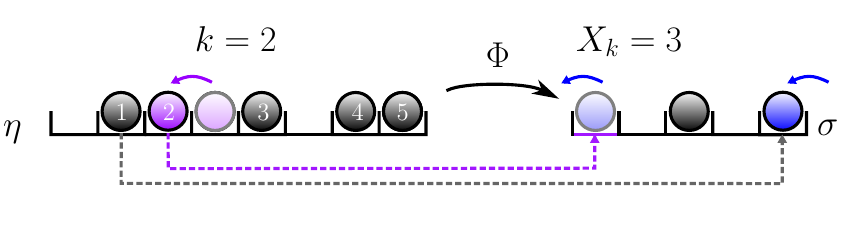}
        \caption{Change of $X_k$ and the current through $\sigma$}\label{fig:saut_courant}
    \end{subfigure}   
        \hfill
		\begin{subfigure}[t]{7.62cm}
		 \vspace{0pt}  
        \includegraphics[width=7.62cm]{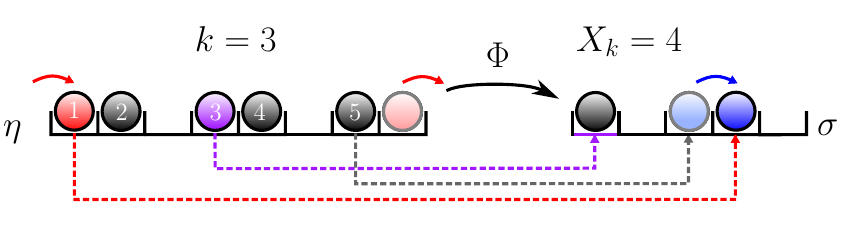}
        \caption{Change of $k$ and the current through $\eta$}\label{fig:saut_rang}
    \end{subfigure}		
    \caption{Different kinds of jumps. The tagged particle is coloured in purple. }\label{fig:ex_sauts}
\end{figure}
\begin{proof}We list every possible transition of $(X,\eta)$ and their effect on $(X,\sigma) = \Phi(k,\eta)$. Notice that, aside from the evolution of $X$, this is the same proof as the dynamic mapping in \cite{erignoux_cutoff_2024}, but we focus here on the ergodic component. The transitions are summarised in Figure \ref{fig:ex_sauts}.

\medskip
\noindent \textit{Jumps that do not affect $X$.} We consider jumps of particles other than the $k^{th}$ particle.

\smallskip
Consider a particle in $\eta$ of rank $k+l \mod K$ with $l\neq 0$, and set $y=x_{k+l}(\eta)$ its position. If it has a right neighbour and an empty site to its left, as in Figure \ref{fig:saut_gauche}, then $\sigma_l = 1$ and $\sigma_{l-1}=0$. 
Indeed, $x_{k+l+1}(\eta)=x_{k+l}(\eta)+1$, and since $\eta$ is ergodic, $x_{k+l-1}(\eta)=x_{k+l}(\eta)-2$. 
Let $\eta^{y,y-1}$ the configuration after the particle jumps to the left: then, $x_{k+l}(\eta^{y,y-1})=x_{k+l}(\eta)-1$, so $x_{k+l+1}(\eta^{y,y-1})=x_{k+l}(\eta^{y,y-1})+2$ and $x_{k+l-1}(\eta^{y,y-1})=x_{k+l}(\eta^{y,y-1})-1$. 
Therefore, $\Phi(k,\eta^{y,y-1}) = (X,\sigma^{l,l-1})$. 
\smallskip

Similarly, consider a particle in $\eta$ of rank $k+l \mod K$ with $l\neq 0$  and position $y$. If it has a left neighbour and an empty site to its right, as in Figure \ref{fig:saut_droite}, then $\sigma_{l-1} = 1$ and $\sigma_{l}=0$. 
Let $\eta^{y,y+1}$ the configuration after the particle jumps to the right: then, $x_{k+l}(\eta^{y,y+1})=x_{k+l}(\eta)+1$, so $x_{k+l+1}(\eta^{y,y+1})=x_{k+l}(\eta^{y,y+1})+1$ and $x_{k+l-1}(\eta^{y,y+1})=x_{k+l}(\eta^{y,y+1})-2$. 
Therefore, $\Phi(k,\eta^{y,y+1}) = (X,\sigma^{l-1,l})$. 

\medskip
\noindent \textit{Jumps that change $X$.} If $\eta_{X+1} = 1$ and $\eta_{X-1} = 0$, as in Figure \ref{fig:saut_courant}, then $\sigma_1 = 1$, $\sigma_K = 0$. As before, the jump $\eta \to \eta^{X,X-1}$ induces a jump $\sigma \to \sigma^{1,K}$. Such a jump reduces the current through $(K,1)$ in the SSEP by $1$. 

If $\eta_{X+1} = 0$ and $\eta_{X-1} = 1$, then $\sigma_1 = 0$, $\sigma_K = 1$. Then, the jump $\eta \to \eta^{X,X+1}$ induces a jump $\sigma \to \sigma^{K,1}$, which increases the current through $(K,1)$ in the SSEP by $1$. 

In both cases, the current through $(K,1)$ in $\sigma$ has the same evolution as the position of the tagged particle $X$. 

\medskip
To conclude, all of these transitions occur at rate $1$, so $(\sigma(t))_{t\ge 0}$ has the law of a SSEP started from $\sigma(0)$, and $(X(t))_{t\ge 0}$ evolves like the current through $(K,1)$ in this SSEP.
\end{proof}

\begin{corollary} \label{cor:mapc}
The trajectory of a tagged particle in an ergodic FEP can be coupled in a deterministic way to the total current through the origin in the corresponding SSEP, taken modulo $N$. 
\end{corollary}

\begin{remark}
This can be generalised to any starting configuration for the FEP: the trajectory of a tagged particle in the FEP can be coupled in a deterministic way to the total current in an associated SSEP with traps, introduced in \cite{erignoux_cutoff_2024}.
\end{remark}

For $\sigma \in \Gamma_{K,2K-N}$, we denote by $\bP_\sigma(\sigma(t) \in \cdot)$ the distribution of the SSEP started from $\sigma$ at time $t$. Set
\eqs{\pi_{K,2K-N}^{\SSEP} = \mathcal{U}(\Gamma_{K, 2K-N})}
its invariant measure. We connect the convergence to stationarity of the SSEP and the FEP through the mapping.
\begin{proposition}[Stationary measures and effect of the mapping] \label{prop:stat_meas_map}
The process $(X(t), \sigma(t))_{t\ge 0}$ is a Markov chain, with stationary measure
\eqs{\nu_{N,K}^{\hbox{\tiny c},\SSEP} := \mathcal{U}(\T_N) \otimes \pi^{\SSEP}_{K,2K-N}.} 

Furthermore, if $(X, \sigma) \sim \mathcal{U}(\T_N) \otimes \pi^{\SSEP}_{K,2K-N}$, then 
\eqs{\Phi^{-1}(X,\sigma) \sim \mathcal{U}(\T_K) \otimes \pi^{\FEP}_{N,K} =: \mu_{N,K}^{\hbox{\tiny r},\FEP}.}
\end{proposition}
\begin{proof}
It is clear that $(X(t), \sigma(t))_{t\ge 0}$ is a Markov chain, of generator:
\begin{align}
\mathcal{L}^{\hbox{\tiny c},\SSEP}&f(x, \sigma) = \sum_{k=1}^{K-1} \left(\sigma_k(1-\sigma_{k+1} )+ \sigma_{k+1}(1-\sigma_k) \right)\left(f(x, \sigma^{k,k+1}) - f(x,\sigma)\right) \nonumber \\
&+ \left(\sigma_K(1-\sigma_{1} )+ \sigma_{1}(1-\sigma_K) \right) \left(f(x + \sigma_{K}-\sigma_{1}, \sigma^{K,1}) - f(x,\sigma)\right).
\end{align}
For all $(x,\sigma)$, $\mathcal{L}^{\hbox{\tiny c},\SSEP} \nu_{N,K}^{\hbox{\tiny c},\SSEP}  (x,\sigma) = 0.$
Last, for any $(k,\eta)$ in $\T_K \times \mathcal{E}_{N,K}$,
\eqs{\nu_{N,K}^{\hbox{\tiny c},\SSEP}\left(\Phi(k,\eta)\right) = \frac{1}{N}\frac{1}{\binom{K}{2K-N}} = \frac{1}{K} \frac{1}{|\mathcal{E}_{N,K}|} = \mu_{N,K}^{\hbox{\tiny r},\FEP}(k,\eta).}

\end{proof}

This result allows to express the distance to equilibrium of an ergodic FEP and its current in terms of the distance to equilibrium of a SSEP and its current.
\section{Proof of the upper bound}\label{sec:u}
We give ourselves an initial ergodic configuration $\eta\in \mathcal{E}_{N,K}$, and an initial rank 
\eqs{k(0) \sim \mathcal{U}(\T_K),} 
independent from all the process. For all $t\ge 0$, set 
\eq{eq:defXsigma}{(X(t),\sigma(t)) = \Phi(k(t),\eta(t)).}
Our strategy will be to find a time such that $(X(t), \sigma(t))$ has probably been coupled with a $(X',\sigma')$ whose law is close enough to $\mathcal{U}(\T_{N}) \otimes \pi^{\SSEP}_{K,2K-N}$. Then, by Proposition \ref{prop:stat_meas_map}, $d_{\textsc{tv}}(\PP_\eta(\eta(t) \in \cdot), \pi^{\FEP}_{N,K})$ will be small. 
To lighten notation, we give a name to the number of particles $2K-N$ in $\sigma$ and set
 \eqs{P = 2K-N.}
The following proofs are formulated for the case where $2K-N \le K/2$, meaning that the SSEP has less particles than empty sites, but all the proofs can be done for $2K-N > K/2$ by replacing all of the $P$ by $N-K$ (this corresponds to viewing the particles in the SSEP as empty sites and vice versa). 
In Section \ref{subsec:hf} we introduce tools from \cite{lacoin_simple_2017}, then in Sections \ref{subsec:coup} to \ref{subsec:init-h} we prove Proposition \ref{prop:UB}, so we assume that $\frac{N}{2} + a_N < K < N-a_N $. We will prove Proposition \ref{prop:UBc}, which covers the cases $K \le \frac{N}{2} + a_N$ and $K \ge N-a_N$, in Section \ref{subsec:UBc}. The proof follows the same steps as for Proposition \ref{prop:UB}, it is in fact a simpler regime, but we need to prove a coupling time that was not studied in \cite{lacoin_simple_2017}.

 \subsection{The height function representation}
 \label{subsec:hf}
 We now introduce the representation of the couple $(X(t), \sigma(t))$ as a height function. This is a very convenient way to keep track of the joint law of the two coordinates. We inspire ourselves from \cite{lacoin_simple_2017} where this is used as a tool to develop a coupling between the SSEP and its equilibrium measure. In the case of \cite{lacoin_simple_2017}, only the SSEP part is looked at, but here we will use to our advantage the fact that when height functions couple, not only the SSEP parts but also the first coordinates have coupled. We hereafter give definitions and some useful properties from \cite{lacoin_simple_2017}.
 
\begin{definition}[Height function associated with a couple] Let $(Y, \sigma) \in \Z \times \Gamma_{K,P}$. The map to a height function is defined as
\begin{equation}
 	\Psi:
    \begin{cases}
        \Z \times \Gamma_{K,P} \to \R^{\T_K}\\
        (Y,\sigma) \mapsto  \zeta,
    \end{cases}       
    \end{equation}
where
 \begin{equation}
    \begin{cases}
        \zeta_0 =\zeta_K = - Y\\
        \zeta_k = \zeta_{k-1} +  \sigma_k-\frac{P}{K} \quad \forall \, 1\le k<K.
    \end{cases}       
    \end{equation}
\end{definition} 
In this construction, when $\sigma$ has a particle at site $k$, $\zeta_k$ goes up by $1-\frac{P}{K}$, otherwise it goes down by $\frac{P}{K}$. Notice that we allow the first coordinate $Y$ to be any integer and do not take it modulo $N$ yet, this will be useful for us to put height functions one above another. Set 
 \eqs{\Omega_{K,P} = \Psi(\Z \times \Gamma_{K,P})}
 the set of possible height functions. Then $\Psi$ is bijective from $\Z \times \Gamma_{K,P}$ to $\Omega_{K,P}$.

We define a Markov chain on $\Omega_{K,P}$ called the corner-flip dynamics, following the definition from \cite[Section 5]{lacoin_simple_2017}.
\begin{definition}[Corner-flip dynamics]
For $\zeta \in \Omega_{K,P}$ and $k \in \T_K$, define $\zeta^k$ the configuration with a flip at $k$ such that 
\begin{equation}
    \begin{cases}
        \zeta^k_l =\zeta_l \quad \forall \, l \neq k\\
        \zeta^k_k = \zeta_{k+1} + \zeta_{k-1} - \zeta_k.
    \end{cases} 
    \end{equation}
    Then, the corner-flip dynamics on $\Omega_{K,P}$ is the Markov chain such that for all $k$, $\zeta$ goes to $\zeta^k$ at rate 1 and other transitions are not possible. 
\end{definition}
This dynamics corresponds to turning a local maximum at $k$ into a local minimum by reducing $\zeta_k$ by 1, and vice versa, see Figure \ref{fig:height_current} for an illustration. Then we have the following property:
 \begin{proposition}[Dynamic mapping with the height function]\label{prop:map_height}
 Consider $(\sigma(t))_{t\ge 0}$ a SSEP started from $\sigma$ and $Y(0) \in \mathbb{Z}$. Set, for all $t$, $Y(t) - Y(0)$ the total current through the edge $(K,1)$ in $(\sigma(s))_{s\le t}$ and $\zeta(t) = \Psi(Y(t),\sigma(t))$. Then $(\zeta(t))_{t\ge 0}$ follows corner-flip dynamics as defined above.
 \end{proposition}
 
 \begin{figure}
 \centering
 \includegraphics[width=0.5\textwidth]{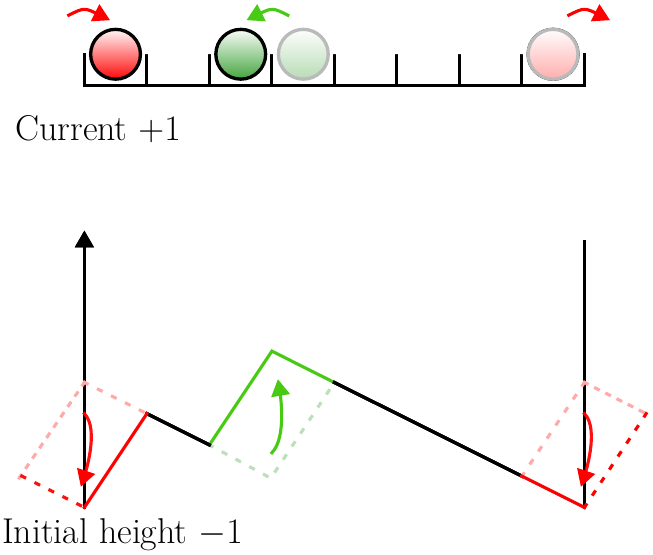}
 \caption{Illustration of the dynamic mapping and the link between current and initial height.}
  \label{fig:height_current}
 \end{figure}
 
 \begin{proof}
We analyse all possible transitions, which are illustrated in Figure \ref{fig:height_current}. It is clear that a particle jump in $\sigma$ induces a corner flip in $\zeta$: going from $\bullet\circ$ to $\circ \bullet$ in the SSEP means we go from ``up-down'' to ``down-up'' in the height function, and vice versa. This is in fact a classical representation for one-dimensional exclusion processes, used for example in \cite{wilson_mixing_2004}. Consider now a jump across the origin. If this jump increases the current $Y$, we go from a local maximum to a local minimum at 0, so $\zeta_0$ decreases by 1. In this case, $\zeta_0$ has the same evolution as $-Y$. Similarly, if a jump decreases the current, it means we go from a local minimum to a local maximum at 0, so that $\zeta_0$ increases by 1, and again $\zeta_0$ has the same evolution as $-Y$. 
\end{proof}

The following properties on the fluctuations of the SSEP will be useful to study the height functions.

 \begin{proposition}[Fluctuations of the density in the SSEP]\cite[Proposition 3.2]{lacoin_simple_2017}
 \label{prop:flu_ssep}
There exists $c>0$ such that for $N$ sufficiently large, for all $\sigma \in \Gamma_{K,P}$ and $s \ge 16$, if $t \ge \frac{1}{8\pi^2}K^2\log P$,
 \eqs{\textnormal{\bP}_\sigma\left(\exists k,l \in \T_K, \left|\sum_{j=k}^l \left(\sigma_j(t) - \frac{P}{K}\right) \right| \ge s \sqrt{P}\right) \le 2 \exp(-cs^2).}
 \end{proposition}
 
  \begin{proposition}[Fluctuations of the density in the stationary SSEP]\cite[Remark 4.3]{lacoin_simple_2017}
There exists $c>0$ such that for $N$ sufficiently large, for all $s \ge 0$, if $\sigma \sim \pi_{K,P}^{\SSEP}$,
 \eqs{\pi_{K,P}^{\SSEP}\left(\exists k,l \in \T_K, \left|\sum_{j=k}^l \left(\sigma_j - \frac{P}{K} \right)\right| \ge s \sqrt{P}\right) \le 2 \exp(-cs^2).}
 \end{proposition}
 
 Last, we will use the coupling and the time needed to couple from \cite{lacoin_simple_2017}. Notice that all height functions of $\Omega_{K,P}$ have their points on the lattice 
 \eqs{\Lambda_{K,P}:=\{(k, m - kP/K), k \in \T_K, m \in \Z\}.} We fix a family of i.i.d. rate 1 Poisson processes $\{\mathcal{T}_{(k,h)}^{\uparrow}, \mathcal{T}_{(k,h)}^{\downarrow}, (k,h) \in \Lambda_{K,P}\}$. There are two clocks per point of the lattice, and in the construction of \cite[Section 5.3]{lacoin_simple_2017}, the clock $\mathcal{T}^{\uparrow}_{(k,h)}$ will prompt upward corner flips from point $(k,h)$, and $\mathcal{T}^{\downarrow}_{(k,h)}$ will prompt downward corner flips from point $(k,h)$. We define more precisely the coupling here. 
\begin{definition} \cite[Section 5.3]{lacoin_simple_2017} One can construct the trajectory of a height function $\zeta$ given the clocks $\{\mathcal{T}_{(k,h)}^{\uparrow/\downarrow}, (k,h) \in \Lambda_{K,P}\}$ by making $(\zeta(t))_{t\ge 0}$ to be càdlàg and to jump only when the $\mathcal{T}_{(k,\zeta_k(t))}^{\uparrow/\downarrow}$ jump, with the following effects:
\begin{itemize}
\item If $\mathcal{T}_{(k,\zeta_k(t^-))}^{\downarrow}$ jumps at time $t$ and if $\zeta_{k-1}(t^-) = \zeta_k(t^-) - 1 = \zeta_{k+1}(t^-)$, then $\zeta(t)=\zeta^k(t^-)$;
\item If $\mathcal{T}_{(k,\zeta_k(t^-))}^{\uparrow}$ jumps at time $t$ and if $\zeta_{k-1}(t^-) = \zeta_k(t^-) + 1 = \zeta_{k+1}(t^-)$, then $\zeta(t)=\zeta^k(t^-)$;
\item In other cases, do nothing.
\end{itemize}
Using the same set of clocks for different height functions defines a coupling between height functions, that we denote by $\Q$.
\end{definition}
We will use the following properties of this coupling.
 \begin{proposition}[Monotone coupling] \label{prop:monot}
 Let $\zeta, \zeta', \zeta'' \in \Omega_{K,P}$ such that $\zeta \le \zeta'\le \zeta''$. Then, under $\Q$, 
 $\zeta(t)$ (resp. $\zeta'(t)$,$\zeta''(t)$) has the law of corner-flip dynamics started from $\zeta$ (resp. $\zeta'$, $\zeta''$) at time $t$, and for all $t \ge 0$, $\zeta(t) \le \zeta'(t)\le \zeta''(t)$ $\Q$-a.s.
 \end{proposition}
 \begin{proposition}[Coupling of two height functions] \label{prop:coupl}
\cite[Proposition 5.3]{lacoin_simple_2017} Assume that $P=P(K)$ goes to infinity with $K$, and $P\le K/2$ for all $K$. Let $\sigma \sim \pi^{\SSEP}_{K,P}$, let $x \in \Z$ and $H \ge 0$. Let $\zeta^{(1)}=\zeta^{(1),H} = \Psi(x+H, \sigma)$ and $\zeta^{(2)}=\zeta^{(2),H} = \Psi(x-H, \sigma)$. Notice that $\zeta^{(1)} \ge \zeta^{(2)}$. 
For all $s>0$ and $\varepsilon \in (0,1)$, there exists $C(s,\varepsilon)>0$ such that if $H \le s \sqrt{P}$, for all $t \ge C(s, \varepsilon) K^2$,
\eqs{\Q_{(\zeta^{(1)}, \zeta^{(2)})}\left(\zeta^{(1)}(t) \neq \zeta^{(2)}(t)\right) \le \varepsilon,}
using the same coupling as before.
 \end{proposition}
 
 \subsection{Coupling with near-equilibrium}
 \label{subsec:coup}
  \begin{figure}
 \centering
 \includegraphics[width=0.9\textwidth]{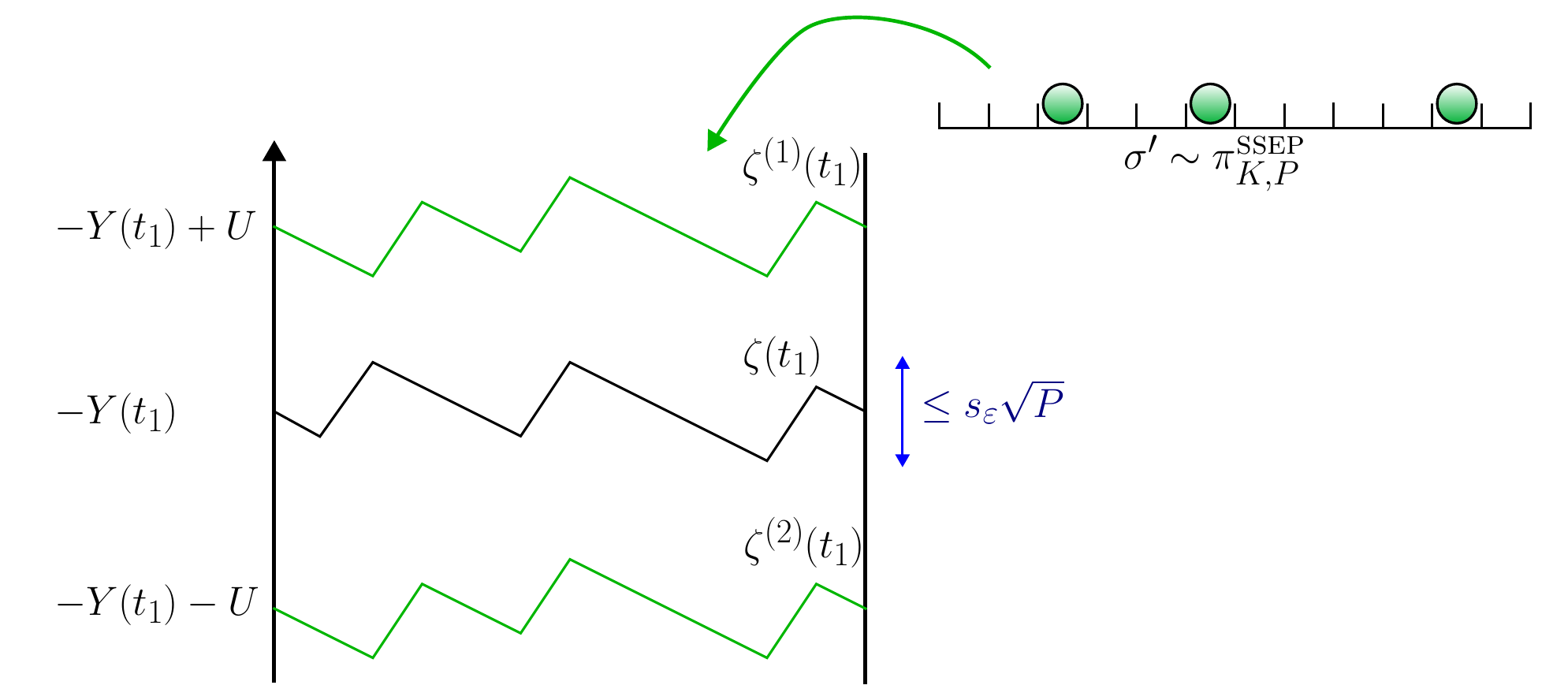}
 \caption{Summary of the coupling strategy. The maximum height difference of $\zeta$ is controlled by Proposition \ref{prop:flu_ssep}. This gives us the inequality $\zeta^{(2)} \le \zeta \le \zeta^{(1)}$. Lemma \ref{lem:init_height} tells us the initial heights of $\zeta^{(1)}$ and $\zeta^{(2)}$ modulo $N$ are close to being uniform. So when all height functions are coupled, $(X(t), \sigma(t))$ is close to $\mathcal{U}(\T_N) \otimes \pi_{K,P}^{\SSEP}$.}
 \label{fig:height_coup}
 \end{figure}
We assume that $P> 2 a_N$ (or $P> a_N$ if $N-K<2K-N$). We wish to find a time such that the height function associated to $\left(X(t),\sigma(t)\right)$,  defined in \eqref{eq:defXsigma}, has been coupled to a height function close to stationarity with probability greater than $1-\varepsilon$.
 The strategy for this coupling is summarised in Figure \ref{fig:height_coup}. Recall we consider, for all $t$, $\left(X(t), \sigma(t)\right)$ where $X(t) \in \T_N$ and $X(t)-X(0)$ gives the total current (modulo $N$) that went through the edge $(K,1)$ in $(\sigma(s))_{s\le t}$. It will be more convenient for our purpose to work with $(Y(t), \sigma(t))$, where $Y(0) = X(0)$ and for all $t$, $Y(t) - Y(0) \in \Z$ is the total current through the origin in $(\sigma(s))_{s\le t}$, not modulo $N$. It is then easy to go back to $X(t)$ by taking $Y(t)$ modulo $N$. 
 
 Now, for all $t$, we set  
 \eqs{\zeta(t)=\Psi(Y(t), \sigma(t))}
  the associated height function. Since we considered the current in $\Z$ and not modulo $N$, we can use the dynamic mapping with the height function from Proposition \ref{prop:map_height}. 

 We set $t_1 = \frac{1}{8\pi^2}K^2\log P$. By Proposition \ref{prop:flu_ssep} we can give ourselves $s=s_\varepsilon$ such that 
\eqs{\bP_\sigma\left(\exists k,l, \left|\sum_{j=k}^l \left(\sigma_j(t_1) - \frac{P}{K}\right)\right| \ge s\sqrt{P}\right) \le \varepsilon/4.}
Notice that this is equivalent to 
 \eq{eq:max_height}{\Q\left(\max\limits_{k,l\in \T_K}\zeta_k(t_1) - \zeta_l(t_1)  > s \sqrt{P}\right) \le \varepsilon/4.}

We therefore wait a first period of length $t_1$ for the maximal amplitude of $\zeta$ to become small. We then define height functions $\zeta^{(1)}$ and $\zeta^{(2)}$ so that at time $t_1$, they surround $\zeta$ with probability greater than $1-\varepsilon$, as illustrated in Figure \ref{fig:height_coup}. 

\begin{definition}\label{def:init_cond}
Take $\sigma' \sim \pi^{\SSEP}_{K,P}$ independent of $k(0)$ and the Poisson clocks, $s' = \frac{2s}{\varepsilon}$ and $U \sim \mathcal{U}\left(\left[ \entc{ s' \sqrt{P}}+1, 2\entc{ s' \sqrt{P} }\right]\right)$ independent of $k(0)$, $\sigma'$ and the Poisson clocks. Then we set
\eqs{\zeta^{(1)}(t_1) = \Psi(Y(t_1) - U,\sigma') \quad\hbox{ and }\quad\zeta^{(2)}(t_1) = \Psi(Y(t_1) + U,\sigma').}\end{definition}
These height functions satisfy the following property:
\begin{lemma}[Inequality between height functions] \label{lem:ineq_height}
With probability greater than $1-\varepsilon$,
\eqs{\zeta^{(2)}(t_1) \le \zeta(t_1) \le \zeta^{(1)}(t_1).}
\end{lemma}
\begin{proof} For any height function $\xi$, we set $\Delta(\xi) = \max\limits_{k,l\in \T_K}\xi_k(t) - \xi_l(t) $ the amplitude of $\xi$. Since $s' \ge 2s_\varepsilon$, 
\begin{align}
\Q\left(\zeta(t_1) \leq \zeta^{(1)}(t_1)\right) & \ge \Q\left(\Delta(\zeta(t_1)) < s \sqrt{P}, \Delta(\zeta^{(1)}(t_1)) < s \sqrt{P}\right). 
\end{align}
So, 
\begin{align*}\Q\left(\zeta(t_1) \not \le \zeta^{(1)}(t_1)\right) &\le \Q\left(\Delta(\zeta(t_1)) \ge s \sqrt{P}\right)+\Q\left(\Delta(\zeta^{(1)}(t_1)) \ge s \sqrt{P}\right)\\
& \le \frac{\varepsilon}{2} \quad \text{by \eqref{eq:max_height},}
\end{align*} 
and we conclude similarly for $\zeta^{(2)}(t_1)$.
\end{proof}
Now, for $t\ge t_1$ we make the height functions evolve simultaneously under the corner-flip dynamics with the coupling $\Q$.
\begin{lemma}[Time to merge] \label{lem:merge}There exists $C(s',\varepsilon)>0$ such that, setting $t_2 = C(s',\varepsilon)K^2$,
\eqs{\Q\left(\zeta^{(1)}(t_1+t_2) = \zeta^{(2)}(t_1+t_2) = \zeta(t_1+t_2)\right) \ge 1- 2\varepsilon}
\end{lemma}
\begin{proof}
This follows from applying Markov property at time $t_1$ and combining Lemma \ref{lem:ineq_height} and Propositions \ref{prop:monot} and \ref{prop:coupl}. 
\end{proof}

We now know that at time $t=t_1+t_2$, the distribution of $(Y(t),\sigma(t))$ is coupled with the distribution of $\Psi^{-1}(\zeta^{(1)}(t))$ with probability $1-2\varepsilon$. All that remains to show is that the latter distribution, taking the first coordinate modulo $N$, is close to $\mathcal{U}(\T_N) \otimes \pi_{K,P}^\SSEP$. This is the object of the following lemmas.

\begin{lemma}[Distribution of the initial height] \label{lem:init_height}
For all $i \in \{1,2\}$, set $X^{(i)}(t_1) = -\zeta^{(i)}_0(t_1) \mod N$. Then,
\eq{eq:dtv_init_h}{d_{\textsc{tv}}\left(\Q(X^{(i)}(t_1) \in \cdot), \mathcal{U}(\T_N)\right) \le \varepsilon.}
\end{lemma}
We postpone the proof of Lemma \ref{lem:init_height} to Section \ref{subsec:init-h}. 

\begin{lemma} \label{lem:lawt1}
For all $t \ge t_1$, set $(Y^{(i)}(t),\sigma^{(i)}(t) ) = \Psi^{-1}\left(\zeta^{(i)}(t)\right)$, and recall that $X^{(i)}(t) = Y^{(i)}(t) \mod N$. Then, for all $t \ge t_1$,
\eqs{d_{\textsc{tv}}\left(\Q\left((X^{(i)}, \sigma^{(i)})(t) \in \cdot \right), \mathcal{U}(\T_N) \otimes \pi^{\SSEP}_{K,P}\right) \le \varepsilon.}
\end{lemma}
\begin{proof}
By Lemma \ref{lem:init_height} and the independence of $X^{(i)}(t_1)$ and $\sigma^{(i)}(t_1)$ (see Definition \ref{def:init_cond}), we have $d_{\textsc{tv}}\left(\Q\left((X^{(i)}, \sigma^{(i)})(t_1) \in \cdot\right), \mathcal{U}(\T_N) \otimes \pi_{K,P}^{\SSEP}\right) \le \varepsilon$. 
Then, under $\Q$, $(X^{(i)}(t), \sigma^{(i)}(t))_{t \ge t_1}$ is distributed as the joint law of a SSEP and its current started from $(X^{(i)}(t_1), \sigma^{(i)}(t_1))$. Since this is a Markov chain, the total variation distance to its invariant measure does not increase. 
\end{proof}

Combining Lemmas \ref{lem:merge} and \ref{lem:lawt1} yields the following proposition.
\begin{proposition}[Distance of $(X(t), \sigma(t))$ to equilibrium] \hspace{1cm}\\ For all $t\ge t_1 + t_2$, 
\eqs{d_{\textsc{tv}}\left(\textnormal{\bP}_{(X(0),\sigma)}\left((X(t), \sigma(t)) \in \cdot\right), \mathcal{U}(\T_N )\otimes \pi_{K,P}^{\SSEP}\right) \le 3\varepsilon.}
\end{proposition}
\begin{proof}
Let $A \subset \T_N \times \Gamma_{K,P}$ and recall that $\nu_{N,K}^{\hbox{\tiny c}, \SSEP} = \mathcal{U}(\T_N) \otimes \pi^{\SSEP}_{K,P}$.
 \begin{align}
 \bP_{(X,\sigma)}\!\left((X(t),\sigma(t)) \in A\right)  &\le \bP_{(X,\sigma)}\!\left((X^{(1)}(t),\sigma^{(1)}(t)) \in A\right) + \Q\left(\zeta^{(1)}(t) \neq \zeta(t)\right) \nonumber  \\
 &\le  \nu_{N,K}^{\hbox{\tiny c}, \SSEP}(A) + \varepsilon + 2 \varepsilon.
 \end{align}

\end{proof}

We conclude the proof of Proposition \ref{prop:UB}, by showing that \eqs{\forall t \ge \frac{K^2}{8\pi^2}\log P + C_{\varepsilon} K^2, \quad d_{\textsc{tv}}\left(\PP_\eta(\eta(t) \in \cdot), \pi_{N,K}^{\FEP}\right) \le \varepsilon.} 
Let $A \subset \mathcal{E}_{N,K}$ and 
\eqs{t \ge t_1+t_2(\varepsilon/3)=\frac{K^2}{8\pi^2}\log P + C( s'_{\varepsilon/3}, \varepsilon/3) K^2,} with $C(s,\varepsilon)$ from Proposition \ref{prop:coupl}. Let $(X^U, \sigma^U) \sim \mathcal{U}(\T_N)\otimes \pi^{\SSEP}_{K,P}$ and $(k^U, \eta^U) \sim \mathcal{U}(\T_K)\otimes \pi^{\FEP}_{N,K}$, recall by Proposition \ref{prop:stat_meas_map} that $\Phi(k^U, \eta^U)$ has the same law as $(X^U, \sigma^U)$. Then:
\begin{align}
\PP_\eta(\eta(t) \in A) &= \PP_{(k(0),\eta)}\big((k(t), \eta(t))\in \T_K\times A\big) \nonumber\\
&= \bP_{(X(0),\sigma)}\big((X(t), \sigma(t))\in \Phi(\T_K\times A)\big) \nonumber\\
&\le \bP\left((X^U, \sigma^U)\in \Phi(\T_K\times A)\right) + \varepsilon \nonumber\\
&= \PP\left((k^U, \eta^U) \in \T_K \times A\right) + \varepsilon \nonumber\\
&= \pi^{\FEP}_{N,K}(A) + \varepsilon,
\end{align}
which finishes the proof.
\subsection{Distribution of the initial height} \label{subsec:init-h}
We now prove Lemma \ref{lem:init_height}. We will only show \eqref{eq:dtv_init_h} for $X^{(2)}(t_1) = Y(t_1) + U \mod N$, since the proof is the same for $Y(t_1) - U$. We will need this useful property.
 \begin{proposition}[Distribution of the first coordinates] \label{prop:dist_k}
 If $k(0) \sim \mathcal{U}(\T_K)$ and is independent from $(\eta(t))_{t\ge 0}$, then for all $t \ge 0$,
 \begin{itemize}
 \item $k(t)\sim \mathcal{U}(\T_K)$ and is independent from $\eta(t)$.
 \item Conditionally on $\eta(t)$, $X(t)$ is uniformly distributed over the occupied sites of $\eta(t)$.
 \end{itemize}
 \end{proposition}

Let $x \in \T_N$. Let $U \sim \mathcal{U}\left(\left[\entc{s'\sqrt{P}}+1, 2\entc{s'\sqrt{P}}\right]\right)$ as in Definition \ref{def:init_cond}, independent of $(k(t_1),\eta(t_1))$. Recall that we consider segments as clockwise modulo $N$ intervals (or modulo $K$ in the context of $\T_K$). We perform a first computation on the law of the initial height at $t_1$ taken modulo $N$.
\begin{align*}
\PP\big(Y(t_1) + U &= x \mod N\big) = \PP\big(X(t_1) + U = x \mod N\big)   \\
&= \!\sum_{u = \entc{s'\sqrt{P}}+1}^{2\entc{s'\sqrt{P}}} \sum_{k = 1}^K \PP\big(k(t)=k, U=u, x_k(\eta(t_1)) = x-u \!\mod N\big).
\end{align*}
Since $U \perp\!\!\!\perp k(t_1)$ and $\left(U,k(t_1)\right)\perp\!\!\!\perp \eta(t_1)$,
\begin{align}
\PP\big(Y(t_1) + U &= x \mod N\big) \nonumber\\
&=\frac{1}{K}\frac{1}{\entc{s'\sqrt{P}}}\sum_{u = \entc{s'\sqrt{P}}+1}^{2\entc{s'\sqrt{P}}} \sum_{k = 1}^K \PP\big(x_k(\eta(t_1)) = x-u \mod N\big)  \nonumber \\
&= \frac{1}{K}\frac{1}{\entc{s'\sqrt{P}}} \sum_{u = \entc{s'\sqrt{P}}+1}^{2\entc{s'\sqrt{P}}}  \PP\big(\eta_{x-u}(t_1) = 1\big) \nonumber \\
&= \frac{1}{K}\frac{1}{\entc{s'\sqrt{P}}} \E\left[|\eta_{|[x-2\entc{s'\sqrt{P}},\, x-\entc{s'\sqrt{P}}-1]}(t_1)|\right] \nonumber \\
&= \frac{1}{K}\frac{1}{|I|} \E\left[|\eta_{|I}(t_1)|\right], \label{eq:probI}
\end{align}
where $I = \big[x-2\entc{s'\sqrt{P}}, x-\entc{s'\sqrt{P}}-1\big]$ and $|\eta_{|I}|$ denotes the number of particles of $\eta$ in $I$. We have connected the law of $X(t_1)+U\mod N$ to the expected density of particles in a segment of the FEP. We now show that this quantity can be formulated in the SSEP representation.

\begin{proposition}[Link between particle density in the SSEP and the FEP] \label{prop:link_fep_ssep}
For all $s, P \ge 1$,
\begin{multline}
\PP_\eta\left(\exists \hbox{ a segment } I \subset \T_N,  \left||\eta_{|I}(t_1) | - \frac{K}{N} |I|\right| \ge s \frac{K}{N} \sqrt{P}\right) \\ \le \textnormal{\bP}_\sigma\left(\exists k,l, \left|\sum_{j=k}^{l} \left(\sigma_j(t_1) - \frac{P}{K} \right) \right|\ge s\sqrt{P} -1\right).
\end{multline}
\end{proposition}
\begin{proof}
Notice that a segment $J$ of size $j$ in $\sigma(t_1)$ corresponds to $j$ consecutive particles in $\eta(t_1)$, which are contained in a segment of size $\sum_{k\in J} 2 - \sigma_k(t_1)$. Indeed, if there is a particle at one site in $\sigma(t_1)$, it means there is no space between the corresponding particle in $\eta(t_1)$ and its right neighbour, so this corresponding particle occupies a space of 1 in $\eta(t_1)$. If there is no particle at a site of $\sigma(t_1)$, it means the corresponding particle in $\eta(t_1)$ is followed by an empty site in $\eta(t_1)$, so it occupies 2 spaces in $\eta(t_1)$. 

So given a segment $J$ in the SSEP, there is a corresponding segment in the FEP, starting with a particle, of length $\sum_{k\in J} \left(2-\sigma_j(t_1)  \right)$. Now, replacing $P$ by $2K-N$, notice that:

\begin{multline}\bP_\sigma\left(\exists k, l : \left|\sum_{j=k}^{l} \left(\sigma_j(t_1) - \frac{2K-N}{K} \right) \right|\ge s\sqrt{P} \right) \\ = \bP_\sigma\left(\exists k, l, \left| \frac{N}{K}(l-k +1) - \sum_{j=k}^{l} \left(2 - \sigma_j(t_1)\right)\right| \ge s\sqrt{P}\right),
\end{multline}
which is  therefore the probability that in the FEP, there exists a segment $I$ starting with a particle such that $\left| |\eta_{|I}(t_1)| - \frac{K}{N} |I| \right| \ge s\frac{K}{N} \sqrt{P}$. 

If there exists a segment $I = [x,y]$ starting with an empty site
such that $\left| |\eta_{|I}(t_1)| - \frac{K}{N} |I| \right| \ge s \frac{K}{N}\sqrt{P}$, then by ergodicity, the segment $I' = [x+1, y]$ starts with a particle, and we have $\left| |\eta_{|I'}(t_1)| - \frac{K}{N} |I'| \right| \ge \frac{K}{N} (s\sqrt{P}-1)$. So, 
\begin{align}
&\,\PP_\eta\left(\exists \hbox{ a segment } I, \left| |\eta_{|I}(t_1)| - \frac{K}{N} |I| \right| \ge s\frac{K}{N} \sqrt{P}\right)
 \nonumber\\
\le &\,\PP_\eta\bigg(\exists \hbox{ a segment } I \hbox{ starting with a particle }, \nonumber\\
& \qquad\qquad\qquad\qquad\qquad\qquad \left| |\eta_{|I}(t_1)| - \frac{K}{N} |I| \right| \ge \frac{K}{N}(s\sqrt{P}-1)\bigg) \nonumber\\
=& \,\bP_\sigma\left(\exists \hbox{ a segment }J \subset \T_K, \left|\sum_{k\in J}\left(\sigma_k(t_1) - \frac{P}{K}\right)\right| \ge s\sqrt{P}-1\right).
 \end{align}

\end{proof}
\begin{remark} \label{rem:sigmaprime}
For the case when $N-K \le 2K-N$, the same result can be shown by considering $\sigma' = 1-\sigma$ and $P' = N-K$. 
\end{remark}
By Proposition \ref{prop:flu_ssep}, we can choose $s$ such that 
\eqs{\bP_\sigma\left(\exists k, l : \left|\sum_{j=k}^{l} \left(\sigma_j(t_1) - \frac{P}{K} \right)  \right|\ge s\sqrt{P} - 1\right) \le \frac{\varepsilon}{4}.} 
Then, by Proposition \ref{prop:link_fep_ssep}, for all segment $I$ in $\T_N$, 
\eq{eq:prob_seg_eta}{\PP_\eta\left(|\eta_{|I}(t_1)| \ge |I|\frac{K}{N} + s\frac{K}{N}\sqrt{P} \right)\le\frac{\varepsilon}{4}.} We conclude the computation of the law of $X^{(2)}(t_1)$ by taking, for all $x \in \T_N$, $I = \big[x-2\entc{s'\sqrt{P}}, x - \entc{s'\sqrt{P}}-1\big]$, and applying \eqref{eq:probI} and \eqref{eq:prob_seg_eta}:
\begin{align}
\PP\big(X^{(2)}(t_1) + U = x\mod N\big) &= \frac{1}{K}\frac{1}{|I|} \E\left[|\eta_{|I}(t_1)|\right] \nonumber\\
&\le \frac{1}{K}\frac{1}{|I|}\left(|I| \frac{K}{N}  + s \frac{K}{N}\sqrt{P} + \frac{\varepsilon}{4} |I| \right)\nonumber\\
&= \frac{1}{N}  + \frac{1}{N} \frac{s \sqrt{P}}{|I|} + \frac{\varepsilon}{4K}\nonumber\\
&\le \frac{1}{N}  + \frac{1}{N} \frac{s}{s'} + \frac{\varepsilon}{2N}. 
\end{align}
Setting $s' = \frac{2}{\varepsilon} s$, we have $\PP\big(X(t_1) + U = x\big) \le \frac{1}{N} + \frac{\varepsilon}{N}$, and therefore $d_{\textsc{tv}}\big(\PP(X(t_1) + U \in \cdot), \mathcal{U}(\T_N)\big) \le \varepsilon$.

\subsection{Upper bound when the number of particles is close to the edges}\label{subsec:UBc}
Assume now that for all $N$, $K \le N/2 + a_N$ or $K \ge N-a_N$.
We will prove Proposition \ref{prop:UBc}, following a similar proof strategy as for Proposition \ref{prop:UB}: we aim to find a time such that height functions have merged. However, we are in a regime where the height functions have much smaller amplitude as before, so we will not wait a first period of time before attempting the coupling. 
To prove Proposition \ref{prop:UBc}, we will show an upper bound as a function of $P$. \begin{proposition} \label{prop:UBc2}
For all $\varepsilon \in (0,1)$, for all $N,K$ such that $N/2 < K < N$ and $\min(2K-N,N-K)\le 2a_N$, there exists $C_\varepsilon > 0$ such that
\eq{eq:propUBc2}{\tau^{\FEP}_{\mathcal{E}_{N,K}}(\varepsilon) \le  C_\varepsilon (KP)^2,}
where $P=\min(2K-N,N-K)$.
\end{proposition} 
This bound is less precise than Proposition \ref{prop:UB} when $P$ is large, but still holds if $P$ is bounded. 

We are still in the setting described at the beginning of Section \ref{sec:u}, 
with $\eta(0) \in \mathcal{E}_{N,K}$ and $k(0)$ the initial rank of the tagged particle uniform in $\T_K$. We study $(X(t),\sigma(t))$ as defined in \eqref{eq:defXsigma}, and consider as in Section \ref{subsec:coup} $Y(0)=X(0)$ and $Y(t)-Y(0)\in\Z$ the total current through the origin in $\left(\sigma(s)\right)_{s\le t}$. We consider the height function $\zeta(t) = \Psi\left(Y(t),\sigma(t)\right)$ as defined in Section \ref{subsec:hf}. We wish to show that after a time $C_\varepsilon (KP)^2$, $\zeta(t)$ has probably been merged with height functions $\zeta^{(1)}$ and $\zeta^{(2)}$ which are initially close to equilibrium.

\paragraph{Initial height} We introduce height functions $\zeta^{(1)}$ and $\zeta^{(2)}$ and show that, for $i\in\{1,2\}$, $\Psi^{-1}(\zeta^{(i)})$ with the first coordinate taken modulo $N$ is close to $\mathcal{U}(\T_N) \otimes \pi_{K,2K-N}^\SSEP$.
\begin{definition}
Take $\sigma' \sim \pi_{K,P}^\SSEP$ independent of $k(0)$ and the Poisson clocks, and take $U \sim \mathcal{U}\left(\left[2P+1, 2P + \left\lceil \frac{4P}{\varepsilon}\right\rceil\right]\right)$ independent of $k(0)$, the Poisson clocks and $\sigma'$. Then we set
\eqs{\zeta^{(1)}(0) = \Psi(Y(0) - U, \sigma') \quad \text{and} \quad \zeta^{(2)}(0) = \Psi(Y(0) + U, \sigma').}
Almost surely,
\eq{eq:ineq-zeta-zero}{\zeta^{(2)}(0) \le \zeta(0) \le \zeta^{(1)}(0).}
\end{definition}
The inequality \eqref{eq:ineq-zeta-zero} comes from the fact that for any height function $\xi \in \Omega_{K,P}$, its maximum amplitude $\Delta(\xi)$ is less than $P$. 

We now show a result analogous to Lemma \ref{lem:init_height} on the distributions of $-Y(0) + U\mod N$ and $-Y(0) - U\mod N$.
\begin{lemma}[Distribution of the initial height] \label{lem:init_height2}
For all $i \in \{1,2\}$, set $X^{(i)}(0) = -\zeta^{(i)}_0(0) \mod N$. Then,
\eq{eq:dtv_init_h2}{d_{\textsc{tv}}\left(\Q(X^{(i)}(0) \in \cdot), \mathcal{U}(\T_N)\right) \le 3\varepsilon/4.}
\end{lemma}

\begin{proof}
By the same proof as for Equation \eqref{eq:probI}, for all $x\in \T_N$, 
\eqs{\PP\big(Y(0)+U = x \mod N\big) = \frac{1}{K}\frac{1}{|I|} \E\left[|\eta_{|I}(0)|\right],} where $I = \big[x-2P-\left\lceil \frac{4P}{\varepsilon}\right\rceil, x-2P-1\big]$.

If $K \le N/2 + a_N$, the density of particles in the FEP should be close to $\frac12$ everywhere. More precisely, for all segment $J$ of length greater than $P$, the maximal number of particles in $J$ is $P + \entc{(|J| - P)/2}$: this is the case if $P$ particles are grouped together and the rest of the configuration alternates between empty sites and particles. Then, since $\frac{1}{K}\le \frac{2}{N}$:
\begin{align} 
\PP(Y(0)+U = x \mod N) &\le \frac{2}{N} \frac{1}{|I|}(P + 1+ (|I| - P)/2)\nonumber \\
& \le \frac{1}{N} + \frac{3\varepsilon}{4N}.
\end{align}
If $K \ge N-a_N$, the density of particles in the FEP is close to 1. In particular, for every segment $J$ of length greater than $2P$, it contains at least $|J|-P$ particles: this happens if all the $P$ empty sites are in $J$. Then, since $\frac{1}{K}\ge \frac{1}{N}$:
\begin{align}
\PP(Y(0)+U = x \mod N) &\ge \frac{1}{N} \frac{1}{|I|}(|I| - P)\nonumber \\
& \ge \frac{1}{N} - \frac{\varepsilon}{4N}.
\end{align}
\end{proof}

By the same proof as for Lemma \ref{lem:lawt1}, we obtain for all $t\ge 0$:
 \eq{eq:loiPpetit}{d_{\textsc{tv}}\left(\left(X^{(2)}(t),\zeta^{(2)}(t)\right),\mathcal{U}(\T_N) \otimes \pi_{K,P}^\SSEP\right) \le \frac{3\varepsilon}{4}.}

\paragraph{Time to couple the height functions}
Let us compute the time for $\zeta^{(1)}$ and $\zeta^{(2)}$ to merge. 
Our approach follows the ideas of \cite{lacoin_simple_2017}, but we look for a coarser estimate so we do not need such a precise argument. Recall the coupling is defined with upwards and downwards clocks at each point of the grid.
Let 
\eqs{\mathcal{A}(t) = \sum_{k\in \T_K} \left(\zeta^{(1)}_k(t) - \zeta^{(2)}_k(t)\right)} the area between the two height functions. At each jump, $\mathcal{A}$ may increase or decrease by 1, and it is absorbed at zero. 
Each local maximum of $\zeta^{(1)}$ that is not merged with $\zeta^{(2)}$ can decrease $\mathcal{A}$ at rate 1 by flipping downwards. Each local minimum of $\zeta^{(1)}$ that is not merged with $\zeta^{(2)}$ can increase $\mathcal{A}$ at rate 1 by flipping upwards. A similar reasoning holds for the corners of $\zeta^{(2)}$. 
If we denote by $\vee$ a local minimum and $\wedge$ a local maximum of a height function, and write that a local minimum (or maximum) centered in $k$ belongs to $\zeta^{(1)}\cap \zeta^{(2)}$ if $(\zeta^{(1)}_{k-1},\zeta^{(1)}_k,\zeta^{(1)}_{k+1}) = (\zeta^{(2)}_{k-1},\zeta^{(2)}_{k},\zeta^{(2)}_{k+1}) $,
one can check that the rate at which $\mathcal{A}$ increases by 1 is:
\eqs{|\{\vee \text{ of } \zeta^{(1)}\}| - |\{\vee \text{ of } \zeta^{(1)}\cap\zeta^{(2)}\}|+ |\{\wedge \text{ of } \zeta^{(2)}\}|  - |\{\wedge \text{ of } \zeta^{(1)}\cap\zeta^{(2)}\}|,}
and the rate at which $\mathcal{A}$ decreases by 1 is:
\eqs{|\{\wedge \text{ of } \zeta^{(1)}\}| - |\{\wedge \text{ of } \zeta^{(1)}\cap\zeta^{(2)}\}|+ |\{\vee \text{ of } \zeta^{(2)}\}|  - |\{\vee \text{ of } \zeta^{(1)}\cap\zeta^{(2)}\}|.}

For any height function of $\Omega_{K,P}$, the number of $\wedge$ is equal to the number of $\vee$, so the jump rates of $\mathcal{A}$ are symmetric. As noticed in \cite[Section 6]{lacoin_simple_2017}, $\mathcal{A}(t)$ is thus a time change of a simple random walk on $\Z$ absorbed at zero. One can therefore couple $\mathcal{A}$ with a continuous time random walk $Z$, started at $\mathcal{A}(0)$, that jumps left or right at rate 1 and is absorbed at zero, such that, for all $t$, 
\eqs{\Q\left(\mathcal{A}(t) > 0\right) \le \Q\left(Z(t) > 0\right).} 
Since $\mathcal{A}(0) \le 12 KP /\varepsilon$, there exists $C_\varepsilon$ such that 
\eqs{\Q\left(Z\left(C_\varepsilon (KP)^2\right)> 0\right) <\varepsilon/4,}
so at this time, the area has probably reached zero, and therefore $\zeta^{(1)},\zeta^{(2)}$ and $\zeta$ have probably merged.  

Combining this with \eqref{eq:loiPpetit} and Proposition \ref{prop:stat_meas_map} allows us to conclude that 
\eqs{\tau^{\FEP}_{\mathcal{E}_{N,K}}(\varepsilon) \le  C_\varepsilon (KP)^2.}

Let us stress that the proof from \cite{lacoin_simple_2017} is much more involved: here, we compared a martingale that jumps at rate between 2 and $4P$ to a random walk that jumps at rate 2, which was sufficient for our purpose because $P$ is small. In the context of \cite{lacoin_simple_2017} $P$ is large, so this simple comparison is not precise enough to obtain a sharp bound on the coupling time, and a multi-scale argument is developed to control precisely the jump rates and show that the merging occurs fast. This is in fact the result that we used in Proposition \ref{prop:coupl} for the case where $P$ is large.

\section{Proof of the lower bound }\label{sec:b}
In this Section, we set $P = \min(2K-N,N-K)$ and assume that $P\to \infty$. It is the minimum between the number of holes and the number of particles in the SSEP representation of size $K$, in particular we always have $P\le K/2$. The proof is formulated for the case where $2K-N \le N-K$. To show the lower bound when $2K-N > N-K$, simply exchange the role of particles and holes in the SSEP representations. We fix a specific initial configuration 
\eq{eq=eta_lb}{\bar{\eta} = \underbrace{\bullet \bullet ... \,\bullet}_{2K-N}\underbrace{\bullet\circ\bullet\circ ...\bullet\circ}_{2(N-K)}.} 
Our aim is to show that for $t < \frac{1}{8 \pi^2} K^2 \log P - C_\varepsilon K^2$, there is an event that has small probability under the invariant measure but occurs with higher probability for $\bar{\eta}(t) $.
To define this event, let us introduce some notation. For all $\eta \in \mathcal{E}_{N,K}$ and $k\in \T_K$, denote by $\Phi^{(k)}(\eta) \in \Gamma_{K,P}$ the second coordinate of $\Phi(k,\eta)$. 
We define the following functions from $\Gamma_{K,P}$ to $\R$: 
\eqs{\begin{cases}\varphi_1:\sigma \mapsto \sum_{k\in \T_K} \sigma_k \cos\left(\frac{2\pi k}{K}\right) \\
\psi_1:\sigma\mapsto \sum_{k\in \T_K} \sigma_k \sin\left(\frac{2\pi k}{K}\right). \\
\end{cases}}

These are eigenfunctions of the generator $\mathcal{L}_{K,P}^\SSEP$, associated to the eigenvalue $-\lambda_1$, where
\eqs{\lambda_1 = 2\left(1-\cos\left(\frac{2\pi}{K}\right)\right).}
We now define our event, by setting for all $s> 0$
\eqs{A_s = \left\{\eta \in \mathcal{E}_{N,K} : \exists k, \, \varphi_1\left(\Phi^{(k)}(\eta)\right)  > s \sqrt{P}\right\}.}

Then, we will show the following lemma:
\begin{lemma} For all $0<\varepsilon<1$, there exist $C_{\varepsilon}>0$ and $s=s_\varepsilon >0 $ such that for $N$ large enough, 
for all $t \le \frac{K^2}{8\pi^2} \log P - C_\varepsilon K^2$, 
\eq{eq:piA1}{\pi_{N,K}^\FEP(A_s) \le \varepsilon'}
\eq{eq:2eps1}{\PP_{\bar{\eta}}(\bar{\eta}(t) \in A_s) > \varepsilon+\varepsilon',}
where $\varepsilon' = \min(\varepsilon, \frac{1-\varepsilon}{2})$.
\end{lemma}

The proof will rely on the fact that the event $A_s$ is invariant by translation, and can therefore be directly expressed as an event on the SSEP. Then the study of the event on the SSEP follows the same ideas as \cite[Section 2]{lacoin_simple_2017}, namely exploiting eigenfunctions to lower-bound the mixing time. 
\paragraph{Properties of the eigenfunctions} Since $\varphi_1$ and $\psi_1$ are eigenfunctions of $\mathcal{L}_{K,P}^\SSEP$ with eigenvalue $-\lambda_1$, for all probability $\mu$ on $\Gamma_{K,P}$, for all $t\ge 0$, for all $\chi_1 \in \{\varphi_1,\psi_1\}$,
\eq{eq:eig_gen}{\mathbf{E}_\mu\left[\chi_1\left(\sigma(t)\right)\right] = e^{-\lambda_1}\mathbf{E}_\mu\left[\chi_1\left(\sigma(0)\right)\right] ,}
with $\bP_\mu\left(\sigma(t)\in\cdot\right)$ being the distribution of a SSEP at time $t$ conditionally to $\sigma(0) \sim \mu$. We will also use an estimate of the variance of these functions taken from \cite[Section 2.2]{lacoin_simple_2017}. For any initial configuration $\sigma \in \Gamma_{K,P}$, for all $K\ge 3$, $t\ge 0$ and $\chi_1\in\{\varphi_1,\psi_1\}$, 
\eq{eq:var_eig}{\mathrm{Var}_\sigma\big(\chi_1\left(\sigma(t)\right)\big) \le 2P.}
This estimate is proved for $\varphi_1$ in \cite[Section 2.2]{lacoin_simple_2017}, using that the maximal jump rate of the process is $2P$ and bounding the quadratic variation of the martingale $\varphi_1\left(\sigma(t)\right)$. The exact same proof can be done for $\psi_1$.
\paragraph{Upper bound for the stationary measure} We first show we can find $s = s_\varepsilon$ such that \eqref{eq:piA1} holds. 
For all $\sigma\in\Gamma_{K,P}$, $k \in \T_K$, define $\tau^k \sigma = \left(\sigma_{k+l}\right)_{l\in \T_K}$ the translation of $\sigma$ by $k$.
Then, for all $\eta\in\mathcal{E}_{N,K}$,
\eqs{\{\Phi^{(k)}\left(\eta\right),k\in \T_K\} = \{\tau^k\Phi^{(1)}\left(\eta\right),k\in \T_K\} .}
Next, for all $\sigma\in\Gamma_{K,P}$,  $k\in \T_K$, by a direct computation,
\eqs{\varphi_1(\tau^k \sigma) = \cos\left(\frac{2\pi k}{K}\right) \varphi_1(\sigma) + \sin\left(\frac{2\pi k}{K}\right) \psi_1(\sigma).}
This enables us to express $A_s$ as an event on only one SSEP. For $\eta \sim \pi_{N,K}^\FEP$, let $\sigma = \Phi^{(1)}(\eta)$, then:
\begin{align}
\eta\in A_s &\iff \exists k, \varphi_1(\tau^k \sigma) > s\sqrt{P} \\
&\iff \exists k, \cos\left(\frac{2\pi k}{K}\right) \varphi_1(\sigma) + \sin\left(\frac{2\pi k}{K}\right) \psi_1(\sigma) > s\sqrt{P}.
\end{align}
For $\eta \sim \pi_{N,K}^\FEP$, $\sigma:= \Phi^{(1)}(\eta)$ has distribution $\pi_{K,P}^\SSEP$,  therefore
\eqs{\pi_{N,K}^\FEP(A_s) = \pi_{K,P}^\SSEP\left( \exists k, \cos\left(\frac{2\pi k}{K}\right) \varphi_1(\sigma) + \sin\left(\frac{2\pi k}{K}\right) \psi_1(\sigma) > s\sqrt{P}\right).}We upper-bound the latter probability, by noticing that if
there exists $k$ such that $ \cos\left(\frac{2\pi k}{K}\right) \varphi_1(\sigma) + \sin\left(\frac{2\pi k}{K}\right) \psi_1(\sigma) > s\sqrt{P}$, then $|\varphi_1(\sigma)|$ or $|\psi_1(\sigma)|$ is greater than $ s\sqrt{P}/2$.
\eqs{\pi^{\FEP}_{N,K}(A_s) \le \pi^\SSEP_{K,P}\left(|\varphi_1(\sigma)|\ge s\sqrt{P}/2\right) + \pi^\SSEP_{K,P}\left(|\psi_1(\sigma)|\ge s\sqrt{P}/2\right).} 
Under $\pi_{K,P}^\SSEP$, the expectations of $\varphi_1(\sigma)$ and $\psi_1(\sigma)$ are zero (this can be shown by applying \eqref{eq:eig_gen} with $\mu = \pi_{K,P}^\SSEP$). Taking $t$ to infinity in \eqref{eq:var_eig}, for $\chi_1\in \{\varphi_1,\psi_1\}$, 
\eqs{\mathrm{Var}_{\pi_{K,P}^\SSEP}\left(\chi_1(\sigma)\right) \le 2P.}
Therefore, by Chebyshev's inequality, for $\chi_1\in \{\varphi_1,\psi_1\}$,\eqs{\pi_{K,P}^\SSEP\left(|\chi_1(\sigma)|\ge s\sqrt{P}/2\right)  \le 8/s^2 .}
By choosing  $s_\varepsilon \ge 4/\sqrt{\varepsilon'}$, we obtain $\pi_{N,K}^\FEP(A_s)\le \varepsilon'$.
\paragraph{Lower bound for $\bar{\eta}(t)$} We now show that, for all $s$, we can find $C_\varepsilon>0$ such that for all $t \le \frac{K^2}{8\pi^2}\log P-C_\varepsilon K^2$ and for $N$ large enough, \eqref{eq:2eps1} holds. 
Set $\bar{\sigma}=\Phi^{(\entc{P/2}+1)}\left(\bar{\eta}\right)$, then
\eqs{\bar{\sigma} = \un_{ \left[-\entc{P/2}+1,\ent{P/2}\vphantom{\frac12}\right] },}
with the interval being a clockwise modulo $K$ interval.
For all $t\ge 0$, set $r(t)$ the rank in $\bar{\eta}(t)$ of the tagged particle that had initially rank $\entc{P/2}+1$ in $\bar{\eta}$, and $\bar{\sigma}(t)=~\Phi^{\left(r(t)\right)}\left(\bar{\eta}(t)\right)$, then $\bar{\sigma}(t)$ is a SSEP at time $t$ started from $\bar{\sigma}$. If $\varphi_1\left(\bar{\sigma}(t)\right) > s\sqrt{P}$, then $\eta(t)\in A_s$. Let us now study $\varphi_1\left(\bar{\sigma}(t)\right)$: by \eqref{eq:eig_gen}, for all $t\ge 0$,
\eq{eq:eig1}{\mathbf{E}_{\bar{\sigma}}\left[\varphi_1(\bar{\sigma}(t))\right] = e^{-\lambda_1 t}\varphi_1(\bar{\sigma}(0)).}
We compute $\varphi_1(\bar{\sigma}(0))$:
\eqs{\varphi_1(\bar{\sigma}(0)) = \frac{\sin(\pi P/K)}{\sin(\pi/K)}\left(\un_{\{P \text{ is odd}\}} + \cos(\pi/K)\un_{\{P \text{ is even}\}} \right) .}
One can check by analysing the function that for all $K$, if $P\le K/2$, 
\eq{eq:phi10}{\varphi_1(\bar{\sigma}(0))\ge P/2.} 
For all $C<\sqrt{P}/2$, at time $t_C = \frac{1}{\lambda_1}\log \big(\sqrt{P}/(2 C)\big)$, combining \eqref{eq:eig1} and \eqref{eq:phi10} yields that \eqs{\mathbf{E}_{\bar{\sigma}}\left[\varphi_1\left(\bar{\sigma}(t_C)\right)\right] \ge C\sqrt{P }.}
We wish to choose $C>s$ so that the expectation of $\varphi_1(\bar{\sigma}(t_C))$ is greater than  $s\sqrt{P}$, which will help making the event $A_s$ more probable. However we also require that $C<\sqrt{P}/2$ so that $t_C>0$ and \eqref{eq:eig1} makes sense. This is why we need that $P\to \infty$: under this assumption, we are free to choose $C = C_\varepsilon$ as large as we want, and for $N$ large enough the argument will be valid. 
We thus take $C > s$, and using \eqref{eq:var_eig} to estimate the variance, we can apply Cantelli's inequality for all $t \le t_C$.
\begin{align}
\PP_{\bar{\eta}}\left(\bar{\eta}(t) \in A_s\right) &\ge \bP_{\bar{\sigma}}\left(\varphi_1(\bar{\sigma}(t))> s\sqrt{P}\right) \nonumber\\
&\ge \bP_{\bar{\sigma}}\left(\varphi_1(\bar{\sigma}(t)) - \mathbf{E}_{\bar{\sigma}}\left[\varphi_1(\bar{\sigma}(t))\right] >-(C-s)\sqrt{P}\right) \nonumber\\
&\ge \frac{(C-s)^2 P}{(C-s)^2 P + \mathrm{Var}_{\bar{\sigma}}\left(\varphi_1(\bar{\sigma}(t))\right)} \nonumber\\
&\ge \frac{1}{1+2/(C-s)^2}.
\end{align}
Choosing $C > s+\sqrt{\frac{2}{\frac{1}{\varepsilon+\varepsilon'}-1}}$ ensures that
\eq{eq:lbprob}{\PP_{\bar{\eta}}(\bar{\eta}(t) \in A_s) > \varepsilon+\varepsilon'.}
Last, since $\lambda_1 = \frac{4\pi^2}{K^2} + \mathcal{O}(1)$, there exists $C_\varepsilon>0$ such that $t_C\ge \frac{K^2}{8\pi^2}\log P - C_\varepsilon K^2$, so for all $t \le\frac{K^2}{8\pi^2}\log P - C_\varepsilon K^2 $, \eqref{eq:2eps1} is verified.

\section*{Acknowledgements} The author would like to thank Paul Chleboun, Clément Erignoux and Cristina Toninelli for helpful comments and discussions. The author also thanks the anonymous referees, whose useful suggestions helped improving the clarity of the article. 

\bibliographystyle{unsrt}
\bibliography{bib_mix}

\end{document}